\documentclass[sigconf]{acmart}
\copyrightyear{2019} 
\acmYear{2019} 
\setcopyright{acmcopyright}
\acmConference[KDD '19] {The 25th ACM SIGKDD Conference on Knowledge Discovery and Data Mining}{August 4--8, 2019}{Anchorage, AK, USA}
\acmBooktitle{The 25th ACM SIGKDD Conference on Knowledge Discovery and Data Mining (KDD '19), August 4--8, 2019, Anchorage, AK, USA}
\acmPrice{15.00}
\acmDOI{10.1145/3292500.3330936}
\acmISBN{978-1-4503-6201-6/19/08}

\usepackage{diagbox}

\usepackage{algorithm}
\usepackage{algorithmic}
\newtheorem{assumption}{Assumption}
\newtheorem{property}{Property}
\newtheorem{problem}{Problem}
\usepackage{cleveref}
\usepackage{balance}
\balance
\crefformat{footnote}{#2\footnotemark[#1]#3}

\begin{document}
\fancyhead{}
%
\title{ADMM for Efficient Deep Learning with Global Convergence}

%

%
\begin{abstract}
Alternating Direction Method of Multipliers (ADMM) has been used successfully in many conventional machine learning applications and is considered to be a useful alternative to Stochastic Gradient Descent (SGD) as a deep learning optimizer. However, as an emerging domain, several challenges remain, including 1) The lack of global convergence guarantees, 2) Slow convergence towards solutions, and 3) Cubic time complexity with regard to feature dimensions. In this paper, we propose a novel optimization framework for deep learning via ADMM (dlADMM) to address these challenges simultaneously. The parameters in each layer are updated backward and then forward so that the parameter information in each layer is exchanged efficiently. The time complexity is reduced from cubic to quadratic in (latent) feature dimensions via a dedicated algorithm design for subproblems that enhances them utilizing iterative quadratic approximations and backtracking. Finally, we provide the first proof of global convergence for an ADMM-based method (dlADMM) in a deep neural network problem under mild conditions. Experiments on benchmark datasets demonstrated that our proposed dlADMM algorithm outperforms most of the comparison methods.
\end{abstract}

%
%

%
\author{Junxiang Wang, Fuxun Yu, Xiang Chen and Liang Zhao} 
\affiliation{%
  \institution{George Mason University}
}
\email{{jwang40,fyu2,xchen26,lzhao9}@gmu.edu}
\maketitle
 \section{Introduction}
 \indent Deep learning has been a hot topic in the machine learning community for the last decade. While conventional machine learning techniques have limited capacity to process natural data in their raw form, deep learning methods are composed of non-linear modules that can learn multiple levels of representation automatically \cite{lecun2015deep}. Since deep learning methods are usually applied in large-scale datasets, such approaches require efficient optimizers to obtain a feasible solution within realistic time limits.\\
 \indent Stochastic Gradient Descent (SGD) and many of its variants are popular state-of-the-art methods for training deep learning models due to their efficiency.  However,  SGD suffers from many limitations that prevent its more widespread use: for example, the error signal diminishes as the gradient is backpropagated (i.e. the gradient vanishes); and SGD is sensitive to poor conditioning, which means a small input can change the gradient dramatically. Recently, the use of the Alternating Direction Method of Multipliers (ADMM) has been proposed as an alternative to SGD. ADMM splits a problem into many subproblems and coordinates them globally to obtain the solution. It has been demonstrated successfully for many machine learning applications  \cite{boyd2011distributed}. The advantages of ADMM are numerous: it exhibits linear scaling as data is processed in parallel across cores; it does not require gradient steps and hence avoids gradient vanishing problems; it is also immune to poor conditioning \cite{taylor2016training}. \\
\indent Even though the performance of the ADMM seems promising, there are still several challenges at must be overcome: \textbf{1. The lack of global convergence guarantees.} Despite the fact that many empirical experiments have shown that ADMM converges in deep learning applications, the underlying theory governing this convergence behavior remains mysterious. This is because a typical deep learning problem consists of a combination of linear and nonlinear mappings, causing optimization problems to be highly nonconvex. This means that traditional proof techniques cannot be directly applied. \textbf{2. Slow convergence towards solutions.} Although ADMM is a powerful optimization framework that can be applied to large-scale deep learning applications, it usually converges slowly to high accuracy, even for simple examples \cite{boyd2011distributed}. It is often the case that ADMM becomes trapped in a modest solution and hence performs worse than SGD, as the experiment described later in this paper in Section \ref{sec:experiment} demonstrates. \textbf{3. Cubic time complexity with regard to feature dimensions.} The implementation of the ADMM is very time-consuming for real-world datasets. Experiments conducted by Taylor et al. found that ADMM required more than 7000 cores to train a neural network with just 300 neurons \cite{taylor2016training}. This computational bottleneck mainly originates from the matrix inversion required to update the weight parameters. Computing an inverse matrix needs further subiterations, and its time complexity is approximately $O(n^3)$, where $n$ is a feature dimension \cite{boyd2011distributed}.\\
\indent In order to deal with these difficulties simultaneously,  in this paper we propose a novel optimization framework for a deep learning Alternating Direction Method of Multipliers (dlADMM) algorithm. Specifically, our new dlADMM algorithm updates parameters first in a backward direction and then forwards. This update approach propagates parameter information across the whole network and accelerates the convergence process. It also avoids the operation of matrix inversion using the quadratic approximation and backtracking techniques, reducing the time complexity from $O(n^3)$ to $O(n^2)$.  Finally,  to the best of our knowledge, we provide the first proof of the global convergence of the ADMM-based method (dlADMM) in a deep neural network problem. The assumption conditions are mild enough for many common loss functions (e.g. cross-entropy loss and square loss) and activation functions (e.g. rectified linear unit (ReLU) and leaky ReLU) to satisfy. Our proposed framework and convergence proof are highly flexible for fully-connected deep neural networks, as well as being easily extendable to other popular network architectures such as Convolutional Neural Networks \cite{krizhevsky2012imagenet} and Recurrent Neural Networks \cite{mikolov2010recurrent}. Our contributions in this paper include:
\begin{itemize}
\item We present a novel and efficient dlADMM algorithm to handle the fully-connected deep neural network problem. The new dlADMM updates parameters in a backward-forward fashion to speed up the convergence process.
\item We propose the use of quadratic approximation and backtracking techniques to avoid the need for matrix inversion as well as reducing the computational cost for large scale datasets. The time complexity of subproblems in dlADMM is reduced from $O(n^3)$ to $O(n^2)$. 
\item  We investigate several attractive convergence properties of  dlADMM. The convergence assumptions are very mild to ensure that most deep learning applications satisfy our assumptions. dlADMM is guaranteed to converge to a critical point globally (i.e., whatever the initialization is) when the hyperparameter is sufficiently large. We also analyze the new algorithm's sublinear convergence rate.
\item We conduct experiments on several benchmark datasets to validate our proposed dlADMM algorithm. The results show that the proposed dlADMM algorithm performs better than most existing state-of-the-art algorithms, including SGD and its variants.
\end{itemize}
\indent The rest of this paper is organized as follows. In Section \ref{sec:related work}, we summarize recent research related to this topic. In Section \ref{sec:algorithm}, we present the new dlADMM algorithm, quadratic approximation, and the backtracking techniques utilized. In Section \ref{sec:convergence}, we introduce the main convergence results for the dlADMM algorithm. The results of extensive experiments conducted to show the convergence, efficiency, and effectiveness of our proposed new dlADMM algorithm are presented in in Section \ref{sec:experiment}, and Section \ref{sec:conclusion} concludes this paper by summarizing the research.
 \section{Related Work}
 \label{sec:related work}
 Previous literature related to this research includes optimization for deep learning models and ADMM for nonconvex problems.\\
 \indent \textbf{Optimization for deep learning models:} The SGD algorithm and its variants play a dominant role in the research conducted by deep learning optimization community. The famous back-propagation algorithm was firstly introduced by Rumelhart et al. to train the neural network effectively \cite{rumelhart1986learning}. Since the superior performance exhibited by AlexNet \cite{krizhevsky2012imagenet} in 2012, deep learning has attracted a great deal of  researchers' attention and many new optimizers based on SGD have been proposed to accelerate the convergence process, including the use of Polyak momentum \cite{polyak1964some}, as well as research on the Nesterov momentum and initialization by Sutskever et al.  \cite{sutskever2013importance}. Adam is the most popular method because it is computationally efficient and requires little tuning \cite{kingma2014adam}. Other well-known methods that incorporate  adaptive learning rates include AdaGrad \cite{duchi2011adaptive}, RMSProp \cite{tielemandivide},  and AMSGrad \cite{
j.2018on}. Recently, the Alternating Direction Method of Multipliers (ADMM) has become popular with researchers due to its excellent scalability \cite{taylor2016training}. However, even though these optimizers perform well in real-world applications, their convergence mechanisms remain mysterious. This is because convergence assumptions are not applicable to deep learning problems, which often require non-differentiable activation functions such as the Rectifier linear unit (ReLU).\\
\textbf{ADMM for nonconvex problems}: The good performance achieved by ADMM over a range wide of convex problems has attracted the attention of many researchers, who have now begun to investigate the behavior of ADMM on nonconvex problems and made significant advances.
 For example, Wang et al. proposed an ADMM to solve multi-convex problems with a convergence guarantee \cite{wang2019multi}, while Wang et al. presented convergence conditions for a coupled objective function that is nonconvex and nonsmooth \cite{wang2015global}. Chen et al.  discussed the use of ADMM to solve problems with quadratic coupling terms \cite{chen2015extended} and Wang et al. studied the convergence behavior of the ADMM for problems with nonlinear equality constraints \cite{wang2017nonconvex}. Even though ADMM has been proposed to solve deep learning applications  \cite{taylor2016training,gao2019incomplete}, there remains a lack theoretical convergence analysis for the application of ADMM to such problems.
 \begin{table}
 \small
 \centering
 \caption{Important Notations and Descriptions}
 \begin{tabular}{cc}
 \hline
 Notations&Descriptions\\ \hline
 $L$& Number of layers.\\
 $W_l$& The weight matrix for the $l$-th layer.\\
 $b_l$& The intercept vector for the $l$-th layer.\\
 $z_l$& The temporary variable of the linear mapping for the $l$-th layer.\\
 $f_l(z_l)$& The nonlinear activation function for the $l$-th layer.\\
 $a_l$& The output for the $l$-th layer.\\
 $x$& The input matrix of the neural network.\\
 $y$& The predefined label vector.\\
 $R(z_L,y)$& The risk function for the $l$-th layer.\\
 $\Omega_l(W_l)$& The regularization term for the $l$-th layer.\\
 $n_l$& The number of neurons for the $l$-th layer.\\
\hline
 \end{tabular}
 \label{tab:notation}
 \end{table}
 \section{The dlADMM Algorithm}
  \label{sec:algorithm}
 We present our proposed dlADMM algorithm in this section. Section \ref{sec:problem} formulates the deep neural network problem, Section \ref{sec:dlADMM} introduces how the dlADMM algorithm works, and the quadratic approximation and backtracking techniques used to solve the subproblems are presented in Section \ref{sec:quadratic}.

\subsection{Problem Formulation}
\label{sec:problem}
 \ \indent \ \indent Table \ref{tab:notation} lists the important notation utilized in this paper. Even though there are many variants of formulations for deep neural networks, a typical neural network is defined by multiple linear mappings and nonlinear activation functions. A linear mapping for the $l$-th layer is composed of a weight matrix $W_l\in \mathbb{R}^{n_l\times n_{l-1}}$ and an intercept vector $b_l\in \mathbb{R}^{n_l}$, where $n_l$ is the number of neurons for the $l$-th layer; a nonlinear mapping for the $l$-th layer is defined by a continuous activation function $f_l(\bullet)$. Given an input $a_{l-1}\in \mathbb{R}^{n_{l-1}}$ from the $(l-1)$-th layer, the $l$-th layer outputs $a_l=f_l(W_la_{l-1}+b_l)$. Obviously, $a_{l-1}$ is nested in $a_{l}=f_l(\bullet)$. By introducing an auxiliary variable $z_l$,  the task of training a deep neural network problem is formulated mathematically as follows:
\begin{problem}
\label{prob:problem 1}
\begin{align*}
     & \min\nolimits_{W_l,b_l,z_l,a_l} R(z_L;y)+\sum\nolimits_{l=1}^L\Omega_l(W_l)\\
     &s.t.z_l=W_la_{l-1}+b_l(l=1,\cdots,L),\ a_l=f_l(z_l) \ \!(l=1,\cdots,L-1)
\end{align*}
\end{problem}
In Problem \ref{prob:problem 1}, $a_0=x\in\mathbb{R}^{n_0}$ is the input of the deep neural network where $n_0$ is the number of feature dimensions, and $y$ is a predefined label vector. $R(z_L;y)$ is a risk function for the $L$-th layer, which is convex, continuous and proper,  and $\Omega_l(W_l)$ is a regularization term for the $l$-th layer, which is also convex, continuous, and proper. Rather than solving Problem \ref{prob:problem 1} directly,  we can relax Problem \ref{prob:problem 1} by adding an $\ell_2$ penalty to address Problem \ref{prob:problem 2} as follows:
\begin{problem}
\label{prob:problem 2}
\begin{align*}
    &\min\nolimits_{W_l,b_l,z_l,a_l}F(\textbf{W},\textbf{b},\textbf{z},\textbf{a})=R(z_L;y)\!+\!\sum\nolimits_{l=1}^L\Omega_l(W_l)\\&+(\nu/2)\sum\nolimits_{l=1}^{L-1}(\Vert z_l-W_la_{l-1}-b_l\Vert^2_2+\Vert a_l-f_l(z_l)\Vert^2_2)\\ &s.t. \ z_L=W_La_{L-1}+b_L
\end{align*}
\end{problem}
where $\textbf{W}=\{W_l\}_{l=1}^{L}$, $\textbf{b}=\{b_l\}_{l=1}^{L}$, $\textbf{z}=\{z_l\}_{l=1}^{L}$, $\textbf{a}=\{a_l\}_{l=1}^{L-1}$ and $\nu>0$ is a tuning parameter. Compared with Problem \ref{prob:problem 1}, Problem \ref{prob:problem 2} has only a linear constraint $z_L=W_La_{L-1}+b_L$ and hence is easier to solve. It is straightforward to show that as $\nu\rightarrow \infty$, the solution of Problem \ref{prob:problem 2} approaches that of Problem \ref{prob:problem 1}.
\subsection{The dlADMM algorithm}
\label{sec:dlADMM}
\indent We introduce the dlADMM algorithm to solve Problem \ref{prob:problem 2} in this section. The traditional ADMM strategy for optimizing  parameters is to start from the first layer and then update parameters in the following layer sequentially \cite{taylor2016training}. In this case, the parameters in the final layer are subject to the parameter update in the first layer.  However, the parameters in the final layer contain important information that can be transmitted  towards the previous layers  to speed up convergence. To achieve this, we propose our novel dlADMM framework, as shown in Figure \ref{fig:framwork overview}. Specifically, the dlADMM algorithm updates parameters in two steps. In the first, the dlADMM begins updating from the $L$-th (final) layer and moves backward toward the first layer. The update order of parameters in the same layer is $a_l\rightarrow z_l\rightarrow b_l\rightarrow W_l$. In the second, the dlADMM reverses the update direction, beginning at the first layer and moving forward toward the $L$-th (final) layer. The update order of the parameters in the same layer is $W_l\rightarrow b_l\rightarrow z_l\rightarrow a_l$. The parameter information for all layers can be exchanged completely by adopting this update approach.\\
\begin{figure}
    \centering
    \includegraphics[width=\linewidth]{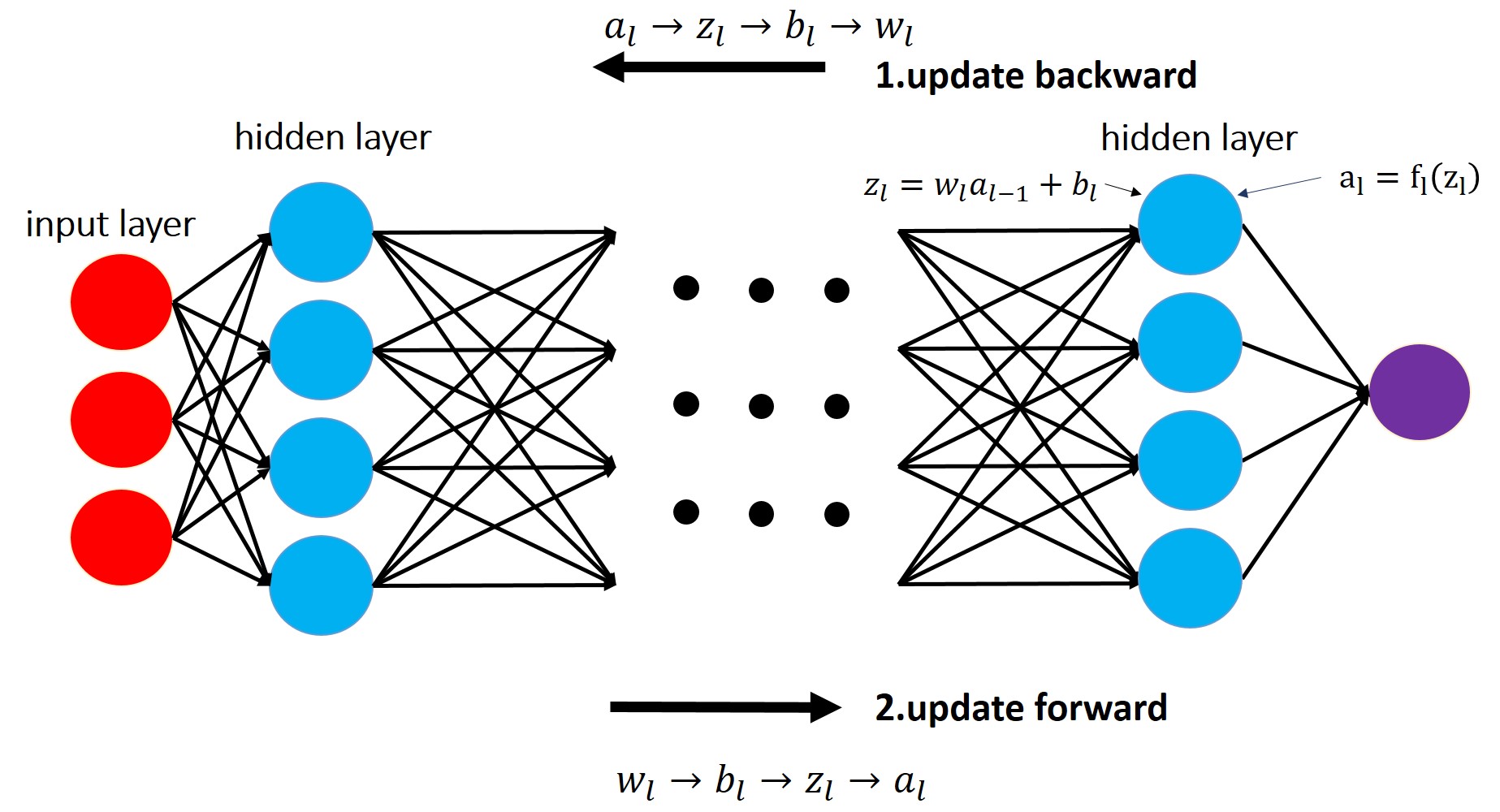}
    \caption{The dlADMM framework overview: update parameter backward and then forward.}
    \label{fig:framwork overview}
\end{figure}
\indent Now we can present our dlADMM algorithm mathematically.
The augmented Lagrangian function of Problem \ref{pro:property 2} is shown in the following form:
\begin{align}
    &L_\rho(\textbf{W}\!,\!\textbf{b},\!\textbf{z},\!\textbf{a},\!u)\!=R(z_L;y)\!+\!\sum\nolimits_{l=1}^L\Omega_l(W_l)\!+\phi(\textbf{W}\!,\!\textbf{b},\!\textbf{z},\!\textbf{a},u)
    \label{eq: Lagrangian}
    \end{align}
where $\phi(\textbf{W}\!,\!\textbf{b},\!\textbf{z},\!\textbf{a},\!u)=\!(\nu/2)\sum\nolimits_{l=1}^{L\!-\!1} (\Vert z_l\!-\!W_l\!a_{l\!-\!1}\!-\!b_l\Vert^2_2+\Vert a_l-f_l(z_l)\Vert^2_2)+u^T(z_L-W_La_{L-1}-b_L)+(\rho/2)\Vert z_L-W_La_{L-1}-b_L\Vert^2_2$, $u$ is a dual variable and $\rho>0$ is a hyperparameter of the dlADMM algorithm. We denote $\overline{W}^{k+1}_l$, $\overline{b}^{k+1}_l$, $\overline{z}^{k+1}_l$ and $\overline{a}^{k+1}_l$ as the backward update of the dlADMM for the $l$-th layer in the $(k+1)$-th iteration , while ${W}^{k+1}_l$, ${b}^{k+1}_l$, ${z}^{k+1}_l$ and ${a}^{k+1}_l$ are denoted as the forward update of the dlADMM for the $l$-th layer in the $(k+1)$-th iteration. Moreover, we denote $\overline{\textbf{W}}^{k+1}_l=\{\{{W}^{k}_i\}_{i=1}^{l-1},\{\overline{{W}}^{k+1}_i\}_{i=l}^L\}$, $\overline{\textbf{b}}^{k+1}_l=\{\{{b}^{k}_i\}_{i=1}^{l-1},\{\overline{{b}}^{k+1}_i\}_{i=l}^L\}$, $\overline{\textbf{z}}^{k+1}_l=\{\{{z}^{k}_i\}_{i=1}^{l-1},\{\overline{{z}}^{k+1}_i\}_{i=l}^L\}$,\\ $\overline{\textbf{a}}^{k+1}_l=\{\{{a}^{k}_i\}_{i=1}^{l-1},\{\overline{{a}}^{k+1}_i\}_{i=l}^{L-1}\}$, ${\textbf{W}}^{k+1}_l=\{\{{W}^{k+1}_i\}_{i=1}^{l},\{\overline{{W}}^{k+1}_i\}_{i=l+1}^{L}\}$, ${\textbf{b}}^{k+1}_l=\{\{{b}^{k+1}_i\}_{i=1}^{l},\{\overline{{b}}^{k+1}_i\}_{i=l+1}^{L}\}$, ${\textbf{z}}^{k+1}_l=\{\{{z}^{k+1}_i\}_{i=1}^{l},\{\overline{{z}}^{k+1}_i\}_{i=l+1}^{L}\}$, ${\textbf{a}}^{k+1}_l=\{\{{a}^{k+1}_i\}_{i=1}^{l},\{\overline{{a}}^{k+1}_i\}_{i=l+1}^{L-1}\}$,
$\overline{\textbf{W}}^{k+1}=\{\overline{W}_i^{k+1}\}_{i=1}^L$, $\overline{\textbf{b}}^{k+1}=\{\overline{b}_i^{k+1}\}_{i=1}^L$,
$\overline{\textbf{z}}^{k+1}=\{\overline{z}_i^{k+1}\}_{i=1}^L$,
$\overline{\textbf{a}}^{k+1}=\{\overline{a}_i^{k+1}\}_{i=1}^{L-1}$, $\textbf{W}^{k+1}=\{W_i^{k+1}\}_{i=1}^L$, $\textbf{b}^{k+1}=\{b_i^{k+1}\}_{i=1}^L$,
$\textbf{z}^{k+1}=\{z_i^{k+1}\}_{i=1}^L$, and
$\textbf{a}^{k+1}=\{a_i^{k+1}\}_{i=1}^{L-1}$. Then the dlADMM algorithm is shown in Algorithm \ref{algo:dlADMM}. Specifically, Lines 5, 6, 10, 11, 14, 15, 17 and 18 solve eight subproblems, namely,  $\overline{a}^{k+1}_l$, $\overline{z}^{k+1}_l$, $\overline{b}^{k+1}_l$, $\overline{W}^{k+1}_l$, ${W}^{k+1}_l$, ${b}^{k+1}_l$, ${z}^{k+1}_l$ and ${a}^{k+1}_l$, respectively. Lines 21 and 22 update the residual $r^{k+1}$ and the dual variable $u^{k+1}$, respectively.
\begin{algorithm} 
\scriptsize
\caption{the dlADMM Algorithm to Solve Problem \ref{prob:problem 2}} 
\begin{algorithmic}[1] 
\label{algo:dlADMM}
\REQUIRE $y$, $a_0=x$, $\rho$, $\nu$. 
\ENSURE $a_l(l=1,\cdots,L-1),W_l(l=1,\cdots,L),b_l(l=1,\cdots,L), z_l(l=1,\cdots,L)$. 
\STATE Initialize $k=0$.
\WHILE{$\textbf{W}^{k+1},\textbf{b}^{k+1},\textbf{z}^{k+1},\textbf{a}^{k+1}$ not converged}
\FOR{$l=L$ to $1$}
\IF{$l<L$}
\STATE Update $\overline{a}^{k+1}_l$ in Equation \eqref{eq:update overline a}.
\STATE Update $\overline{z}^{k+1}_l$ in Equation \eqref{eq:update overline z}.
\STATE Update $\overline{b}^{k+1}_l$ in Equation \eqref{eq:update overline b}.
\ELSE
\STATE{Update $\overline{z}^{k+1}_L$ in Equation \eqref{eq:update overline zl}}.
\STATE Update $\overline{b}^{k+1}_L$ in Equation \eqref{eq:update overline bl}.
\ENDIF
\STATE Update $\overline{W}^{k+1}_l$ in Equation \eqref{eq:update overline W}.
\ENDFOR
\FOR{$l=1$ to $L$}
\STATE Update ${W}^{k+1}_l$ in Equation \eqref{eq:update W}.
\IF{$l< L$}
\STATE Update $b^{k+1}_l$ in Equation \eqref{eq:update b}.
\STATE Update $z^{k+1}_l$ in Equation \eqref{eq:update z}.
\STATE Update $a^{k+1}_l$ in Equation \eqref{eq:update a}.
\ELSE
\STATE Update $b^{k+1}_l$ in Equation \eqref{eq:update bl}.
\STATE{ Update $z^{k+1}_L$ in Equation \eqref{eq:update zl}.}
\STATE{$r^{k+1}\leftarrow z_L^{k+1}-W^{k+1}_L a^{k+1}_{L-1}-b^{k+1}_L$}.
 \STATE{$u^{k+1}\leftarrow u^k+\rho  r^{k+1}$}.
\ENDIF
\ENDFOR
\STATE $k\leftarrow k+1$.
\ENDWHILE
\STATE Output $\textbf{W},\textbf{b},\textbf{z},\textbf{a}$.
\end{algorithmic}
\end{algorithm}
\\
\subsection{The Quadratic Approximation and Backtracking}
\label{sec:quadratic}
\indent The eight subproblems in Algorithm \ref{algo:dlADMM} are discussed in detail in this section. Most can be solved by quadratic approximation and the backtracking techniques described above, so the operation of the matrix inversion can be avoided.\\ \\
\textbf{1. Update $\overline{a}^{k+1}_l$}\\
\indent The variables $\overline{a}^{k+1}_l(l=1,\cdots,L-1)$ are updated as follows:
\begin{align*}
    \overline{a}^{k\!+\!1}_{l}&\leftarrow\! \arg\min\nolimits_{a_{l}}\! L_\rho(\overline{\textbf{W}}^{k\!+\!1}_{l\!+\!1},\overline{\textbf{b}}^{k\!+\!1}_{l\!+\!1},\overline{\textbf{z}}^{k\!+\!1}_{l\!+\!1},\{a^k_i\}_{i\!=\!1}^{l\!-\!1},a_l,\{\overline{a}^{k\!+\!1}_i\}_{i\!=\!l\!+\!1}^{L\!-\!1},u^k)
\end{align*}
\indent The subproblem is transformed into the following form after it is replaced by Equation \eqref{eq: Lagrangian}. 
\begin{align}
    \overline{a}^{k\!+\!1}_{l}\leftarrow\! \arg\!\min\nolimits_{a_{l}}\! \phi\!(\overline{\textbf{W}}^{k\!+\!1}_{l\!+\!1},\overline{\textbf{b}}^{k\!+\!1}_{l\!+\!1},\overline{\textbf{z}}^{k\!+\!1}_{l\!+\!1},\{a^k_i\}_{i\!=\!1}^{l\!-\!1}\!,a_l,\{\overline{a}^{k\!+\!1}_i\}_{i\!=\!l\!+\!1}^{L\!-\!1}\!,\!u^k)
    \label{eq:original overline a}
\end{align}
\indent Because $a_l$ and $W_{l+1}$ are coupled in $\phi(\bullet)$, in order to solve this problem, we must compute the inverse matrix of $\overline{W}^{k+1}_{l+1}$, which involves subiterations and is computationally expensive \cite{taylor2016training}. In order to handle this challenge,  we define $\overline{Q}_l(a_l;\overline{\tau}^{k+1}_l)$ as a quadratic approximation of $\phi$ at $a^{k}_{l}$, which is mathematically reformulated as follows:
\begin{align*}\overline{Q}_l(a_{l};\overline{\tau}^{k+1}_l)&=\phi(\overline{\textbf{W}}^{k+1}_{l+1},\overline{\textbf{b}}^{k+1}_{l+1},\overline{\textbf{z}}^{k+1}_{l+1},\overline{\textbf{a}}^{k+1}_{l+1},u^k)\\&+(\nabla _{a^k_{l}}\phi)^T(\overline{\textbf{W}}^{k+1}_{l+1},\overline{\textbf{b}}^{k+1}_{l+1},\overline{\textbf{z}}^{k+1}_{l+1},\overline{\textbf{a}}^{k+1}_{l+1},u^k)(a_{l}-a^k_{l})\\&+\Vert\overline{\tau}^{k+1}_l\circ (a_{l}-a^k_{l})^{\circ 2}\Vert_{1}/2
\end{align*}
where $\overline{\tau}^{k+1}_l>0$ is a  parameter vector, $\circ$ denotes the Hadamard product (the elementwise product), and $a^{\circ b}$  denotes $a$ to the Hadamard power of $b$ and $\Vert\bullet\Vert_{1}$ is the $\ell_{1}$ norm. $\nabla _{a^k_{l}}\phi$ is the gradient of $\overline{a}_l$ at $a^k_l$. Obviously, $\overline{Q}_l(a^k_l;\overline{\tau}^{k+1}_l)=\phi(\overline{\textbf{W}}^{k+1}_{l+1},\overline{\textbf{b}}^{k+1}_{l+1},\overline{\textbf{z}}^{k+1}_{l+1},\overline{\textbf{a}}^{k+1}_{l+1},u^k)$. Rather than minimizing the original problem in Equation \eqref{eq:original overline a}, we instead solve the following problem:
\begin{align} \overline{a}^{k+1}_l&\leftarrow\arg\min\nolimits_{a_l} \overline{Q}_l(a_l;\overline{\tau}^{k+1}_l) \label{eq:update overline a}
\end{align} Because $\overline{Q}_l(a_{l};\overline{\tau}^{k+1}_l)$ is a quadratic function with respect to $a_{l}$, the solution can be obtained by
\begin{align*}
\overline{a}^{k+1}_{l}\leftarrow a^k_{l}-\nabla_{a^k_{l}}\phi/\overline{\tau}^{k+1}_{l}
\end{align*}
 given a suitable $\overline{\tau}^{k+1}_l$. Now the main focus is how to choose $\overline{\tau}^{k+1}_l$. Algorithm  \ref{algo:overline tau update} shows the backtracking algorithm utilized to find a suitable $\overline{\tau}_l^{k+1}$. Lines 2-5 implement a while loop until  the condition $\phi(\overline{\textbf{W}}^{k+1}_{l+1},\overline{\textbf{b}}^{k+1}_{l+1},\overline{\textbf{z}}^{k+1}_{l+1},\overline{\textbf{a}}^{k+1}_l,u^k)\leq\overline{Q}_l(\overline{a}^{k+1}_l;\overline{\tau}^{k+1}_l)$ is satisfied. As $\overline{\tau}^{k+1}_l$ becomes larger and larger, $\overline{a}^{k+1}_l$ is close to $a^k_l$ and $a^k_l$ satisfies the loop condition, which precludes the possibility of the infinite loop. The time complexity of Algorithm \ref{algo:overline tau update} is $O(n^2)$, where $n$ is the number of features or neurons.
 \begin{algorithm} 
 \scriptsize
\caption{The Backtracking Algorithm  to update $\overline{a}^{k+1}_{l}$ } 
\begin{algorithmic}[1]
\label{algo:overline tau update}
\REQUIRE $\overline{\textbf{W}}^{k+1}_{l+1}$, $\overline{\textbf{b}}^{k+1}_{l+1}$, $\overline{\textbf{z}}^{k+1}_{l+1}$, $\overline{\textbf{a}}^{k+1}_{l+1}$, $u^k$, $\rho$, some constant $\overline{\eta}>1$. 
\ENSURE $\overline{\tau}^{k+1}_l$,$\overline{a}^{k+1}_{l}$. 
\STATE Pick up $\overline{t}$ and $\overline{\beta}=a^k_l-\nabla_{a^k_l}\phi/\overline{t}$
\WHILE{$\phi(\overline{\textbf{W}}^{k+1}_{l+1},\overline{\textbf{b}}^{k+1}_{l+1}\!,\!\overline{\textbf{z}}^{k+1}_{l+1},\{a^k_i\}_{i=1}^{l-1},\overline{\beta},\{\overline{a}^{k+1}_i\}_{i=l+1}^{L-1},u^k)>\overline{Q}_l(\overline{\beta};\overline{t})$}
\STATE $\overline{t}\leftarrow \overline{t}\overline{\eta}$.\\
\STATE $\overline{\beta}\leftarrow a^k_{l}-\nabla_{a^k_{l}}\phi/\overline{t}$.\\
\ENDWHILE
\STATE Output $\overline{\tau}^{k+1}_l \leftarrow \overline{t} $.\\
\STATE Output $\overline{a}^{k+1}_{l}\leftarrow \overline{\beta}$.
\end{algorithmic}
\end{algorithm}
\\
\textbf{2. Update $\overline{z}^{k+1}_l$}\\
\indent The variables $\overline{z}^{k+1}_l(l=1,\cdots,L)$ are updated as follows:
\begin{align*}
    \overline{z}^{k\!+\!1}_{l}&\leftarrow \arg\min\nolimits_{z_{l}} L_\rho(\overline{\textbf{W}}^{k\!+\!1}_{l\!+\!1},\overline{\textbf{b}}^{k\!+\!1}_{l\!+\!1},\{z^k_i\}_{i=1}^{l-1},z_l,\{\overline{z}^{k+1}_i\}_{i=l+1}^{L},\overline{\textbf{a}}^{k+1}_{l},u^k)
\end{align*}
which is equivalent to the following forms: for $\overline{z}^{k+1}_l(l=1,\cdots,L-1)$,
\begin{align}
          \overline{z}^{k\!+\!1}_{l}&\leftarrow\! \arg\min\nolimits_{z_{l}}\! \phi\!(\overline{\textbf{W}}^{k\!+\!1}_{l\!+\!1}\!,\!\overline{\textbf{b}}^{k\!+\!1}_{l\!+\!1}\!,\!\{z^k_i\}_{i\!=\!1}^{l\!-\!1}\!,\!z_l,\{\overline{z}^{k\!+\!1}_i\}_{i\!=\!l\!+\!1}^{L}\!,\overline{\textbf{a}}^{k\!+\!1}_{l}\!,\!u^k)\label{eq:update overline z}
          \end{align}
          and for $\overline{z}^{k+1}_L$,
          \begin{align}
    \overline{z}^{k+1}_L &\leftarrow\arg\min\nolimits_{z_L}     \! \phi\!(\textbf{W}^{k}\!,\!\o\textbf{b}^{k}\!,\!\{z^k_i\}_{i\!=\!1}^{L\!-\!1}\!,\!z_L,\textbf{a}^{k}\!,\!u^k)+R(z_L;y)\label{eq:update overline zl}
\end{align}
\indent Equation \eqref{eq:update overline z} is highly nonconvex because the nonlinear activation function $f(z_l)$ is contained in $\phi(\bullet)$. For common activation functions such as the Rectified linear unit (ReLU) and leaky ReLU, Equation \eqref{eq:update overline z} has a closed-form solution; for other activation functions like sigmoid and hyperbolic tangent (tanh), a look-up table is recommended \cite{taylor2016training}.\\ \indent Equation \eqref{eq:update overline zl} is a convex problem because $\phi(\bullet)$ and $R(\bullet)$ are convex with regard to $z_L$. Therefore, Equation \eqref{eq:update overline zl} can be solved by Fast Iterative Soft-Thresholding Algorithm (FISTA) \cite{beck2009fast}.\\
\textbf{3. Update $\overline{b}^{k+1}_l$}\\
\indent The variables $\overline{b}^{k+1}_l(l=1,\cdots,L)$ are updated as follows:
\begin{align*}
    \overline{b}^{k+1}_l\leftarrow \arg\min\nolimits_{b_l} 
     \!L_\rho(\overline{\textbf{W}}^{k+1}_{l+1},\{\!b^{k}_i\!\}_{i\!=\!1}^{l\!-\!1},b_l,\{\!\overline{b}^{k+1}_i\!\}_{i\!=\!l+1}^{L},\overline{\textbf{z}}^{k+1}_l,\overline{\textbf{a}}^{k+1}_l,u^k)
\end{align*}
which is equivalent to the following form:
\begin{align*}
        \overline{b}^{k+1}_l\leftarrow \arg\min\nolimits_{b_l} 
     \!\phi(\overline{\textbf{W}}^{k+1}_{l+1},\{\!b^{k}_i\!\}_{i\!=\!1}^{l\!-\!1}\!,b_l,\{\!\overline{b}^{k+1}_i\!\}_{i\!=\!l+1}^{L}\!,\!\overline{\textbf{z}}^{k+1}_l,\!\overline{\textbf{a}}^{k+1}_l,u^k)
\end{align*}
\indent Similarly to the update of $\overline{a}^{k+1}_l$, we define $\overline{U}_l(b_l;\overline{B})$ as a quadratic approximation of $\phi(\bullet)$ at ${b}^k_l$, which is formulated mathematically as follows \cite{beck2009fast}:
\begin{align*}
 \overline{U}_l(b_l;\overline{B})&=\phi(\overline{\textbf{W}}^{k+1}_{l+1},\overline{\textbf{b}}^{k+1}_{l+1},\overline{\textbf{z}}^{k+1}_l,\overline{\textbf{a}}^{k+1}_l,u^k)\\&+(\nabla_{b^k_l}\phi)^T(\overline{\textbf{W}}^{k+1}_{l+1},\overline{\textbf{b}}^{k+1}_{l+1},\overline{\textbf{z}}^{k+1}_l,\overline{\textbf{a}}^{k+1}_l,u^k)(b_l-b^k_l)\\&+(\overline{B}/2)\Vert b_l-b^k_l\Vert^2_2.
\end{align*}
where $\overline{B}>0$ is a parameter. Here $\overline{B}\geq \nu$ for $l=1,\cdots,L-1$ and $\overline{B}\geq \rho$ for $l=L$ are required for the convergence analysis  \cite{beck2009fast}. Without loss of generality, we set $\overline{B}=\nu$, and solve the subsequent subproblem as follows:
\begin{align}
    &\overline{b}^{k+1}_l\leftarrow \arg\min\nolimits_{b_l} \overline{U}_l(b_l;\nu)(l=1,\cdots,L-1) \label{eq:update overline b}\\
    &\overline{b}^{k+1}_L\leftarrow \arg\min\nolimits_{b_L} \overline{U}_L(b_L;\rho) \label{eq:update overline bl}
\end{align}
\indent Equation \eqref{eq:update overline b} is a convex problem and has a closed-form solution as follows:
\begin{align*}
    &\overline{b}^{k+1}_{l}\leftarrow b^k_{l}-\nabla_{b^k_{l}}\phi/\nu.(l=1,\cdots,L-1)\\&\overline{b}^{k+1}_{L}\leftarrow b^k_{L}-\nabla_{b^k_{L}}\phi/\rho.
\end{align*}
\textbf{4. Update $\overline{W}^{k+1}_l$}\\
\indent The variables $\overline{W}^{k+1}_l(l=1,\cdots,L)$ are updated as follows:
\begin{align*}
    \overline{W}^{k\!+\!1}_l\!\leftarrow\!\arg\!\min\nolimits_{W_l}\! L_\rho(\{{W}^{k}_i\}_{i\!=\!1}^{l\!-\!1}\!,W_l,\{\!\overline{W}^{k+1}_i\!\}_{i\!=\!l\!+\!1}^{L},\overline{\textbf{b}}^{k\!+\!1}_l,\overline{\textbf{z}}^{k\!+\!1}_l,\overline{\textbf{a}}^{k\!+\!1}_l,u^k)
\end{align*}
which is equivalent to the following form:
\begin{align}
 \nonumber&\overline{W}^{k\!+\!1}_l\!\leftarrow\!\arg\!\min\nolimits_{W_l}\! \phi(\{\!{W}^{k}_i\!\}_{i\!=\!1}^{l\!-\!1},W_l,\{\!\overline{W}^{k+1}_i\!\}_{i\!=\!l\!+\!1}^{L},\overline{\textbf{b}}^{k\!+\!1}_l,\overline{\textbf{z}}^{k\!+\!1}_l,\overline{\textbf{a}}^{k\!+\!1}_l,u^k)\\&+\!\Omega(W_l)
    \label{eq:update W overline original}
\end{align}
\indent Due to the same challenge in updating $\overline{a}^{k+1}_l$,  we define $\overline{P}_l(W_l;\overline{\theta}^{k+1}_l)$ as a quadratic approximation of $\phi$ at ${W}^{k}_{l}$. The quadratic approximation is mathematically reformulated as follows \cite{beck2009fast}:
\begin{align*}\overline{P}_l(W_l;\overline{\theta}^{k+1}_l)&=\phi(\overline{\textbf{W}}^{k+1}_{l+1},\overline{\textbf{b}}^{k+1}_l,\overline{\textbf{z}}^{k+1}_l,\overline{\textbf{a}}^{k+1}_l,u^k)\\&+(\nabla_{W^k_l}\phi)^T(\overline{\textbf{W}}^{k+1}_{l+1},\overline{\textbf{b}}^{k+1}_l,\overline{\textbf{z}}^{k+1}_l,\overline{\textbf{a}}^{k+1}_l,u^k)(W_l-W^k_l)\\&+\Vert\overline{\theta}^{k+1}_l\circ (W_l-W^k_l)^{\circ 2}\Vert_{1}/2
\end{align*}
where $\overline{\theta}^{k+1}_l>0$ is a  parameter vector, which is chosen by the Algorithm \ref{algo:overline theta update}. Instead of minimizing the Equation \eqref{eq:update W overline original}, we minimize the following:
\begin{align}
 &\overline{W}^{k+1}_l \leftarrow \arg\min\nolimits_{W_l} \overline{P}_l(W_l;\overline{\theta}^{k+1}_l)+\Omega_l (W_l)     \label{eq:update overline W}
\end{align}
\indent Equation \eqref{eq:update overline W} is convex and hence can be solved exactly. If $\Omega_l$ is either an $\ell_1$ or an $\ell_2$ regularization term, Equation \eqref{eq:update overline W} has a closed-form solution.
\begin{algorithm} 
\scriptsize
\caption{The Backtracking Algorithm  to update $\overline{W}^{k+1}_{l}$ }
\begin{algorithmic}[1]
\label{algo:overline theta update}
\REQUIRE $\overline{\textbf{W}}^{k+1}_{l+1}$,$\overline{\textbf{b}}^{k+1}_{l}$, $\overline{\textbf{z}}^{k+1}_{l}$,$\overline{\textbf{a}}^{k+1}_{l}$,$u^k$, $\rho$, some constant $\overline{\gamma}>1$. 
\ENSURE $\overline{\theta}^{k+1}_l$,$\overline{{W}}^{k+1}_{l}$. 
\STATE Pick up $\overline{\alpha}$ and $\overline{\zeta}=W^k_l-\nabla_{W^k_l}\phi/\overline{\alpha}$.
\WHILE{$\phi(\{W^k_i\}_{i=1}^{l-1},\overline{\zeta},\{\overline{W}^{k\!+\!1}_i\}_{i\!=\!l\!+\!1}^{L},\overline{\textbf{b}}^{k\!+\!1}_{l},\overline{\textbf{z}}^{k\!+\!1}_{l},\overline{\textbf{a}}^{k\!+\!1}_{l},u^k)>\overline{P}_l(\overline{\zeta};\overline{\alpha})$}
\STATE $\overline{\alpha}\leftarrow \overline{\alpha}\ \overline{\gamma}$.\\
\STATE Solve $\overline{\zeta}$ by Equation \eqref{eq:update overline W}.\\
\ENDWHILE
\STATE Output $\overline{\theta}^{k+1}_l \leftarrow \overline{\alpha} $.\\
\STATE Output $\overline{W}^{k+1}_{l}\leftarrow \overline{\zeta}$.
\end{algorithmic}
\end{algorithm}
\\
\textbf{5. Update ${W}^{k+1}_l$}\\
\indent The variables ${W}^{k+1}_l(l=1,\cdots,L)$ are updated as follows:
\begin{align*}
    W^{k\!+\!1}_l\!\leftarrow\!\arg\!\min\nolimits_{W_l} \!L_\rho\!(\{\!{W}^{k\!+\!1}_i\!\}_{i\!=\!1}^{l\!-\!1}\!,W_l,\{\!\overline{W}^{k\!+\!1}_i\!\}_{i\!=\!l\!+\!1}^{L}\!,\!\textbf{b}^{k\!+\!1}_{l\!-\!1},\textbf{z}^{k\!+\!1}_{l-1},\!\textbf{a}^{k\!+\!1}_{l\!-\!1}\!,\!u^k)
\end{align*}
which is equivalent to 
\begin{align}
 \nonumber&{W}^{k\!+\!1}_l\!\leftarrow\!\arg\!\min\nolimits_{W_l}\! \phi\!(\{\!{W}^{k\!+\!1}_i\!\}_{i\!=\!1}^{l\!-\!1}\!,W_l,\{\!\overline{W}^{k\!+\!1}_i\!\}_{i\!=\!l\!+\!1}^{L}\!,\!\textbf{b}^{k\!+\!1}_{l\!-\!1},\!\textbf{z}^{k\!+\!1}_{l\!-\!1},\!\textbf{a}^{k\!+\!1}_{l\!-\!1}\!,\!u^k)\\&+\!\Omega(W_l)
    \label{eq:update W original}
\end{align}
\indent Similarly,  we define $P_l(W_l;\theta^{k+1}_l)$ as a quadratic approximation of $\phi$ at $\overline{W}^{k+1}_{l}$. The quadratic approximation is then mathematically reformulated as follows \cite{beck2009fast}:
\begin{align*}{P}_l(W_l;{\theta}^{k+1}_l)&=\phi({\textbf{W}}^{k+1}_{l-1},{\textbf{b}}^{k+1}_{l-1},{\textbf{z}}^{k+1}_{l-1},{\textbf{a}}^{k+1}_{l-1},u^k)\\&+(\nabla_{\overline{W}^{k+1}_l}\phi)^T({\textbf{W}}^{k+1}_{l-1},\textbf{b}^{k+1}_{l-1},\textbf{z}^{k+1}_{l-1},{\textbf{a}}^{k+1}_{l-1},u^k)(W_l-\overline{W}^{k+1}_l)\\&+\Vert{\theta}^{k+1}_l\circ (W_l-\overline{W}^{k+1}_l)^{\circ 2}\Vert_{1}/2
\end{align*}
where $\theta^{k+1}_l>0$ is a  parameter vector. Instead of minimizing the Equation \eqref{eq:update W original}, we minimize the following:
\begin{align}
 &{W}^{k+1}_l \leftarrow \arg\min\nolimits_{W_l} P_l(W_l;{\theta}^{k+1}_l)+\Omega_l (W_l)     \label{eq:update W}
\end{align}
The choice of $\theta^{k+1}_l$ is discussed in the supplementary materials\footnote{\label{note} The supplementary materials are available at \url{http://mason.gmu.edu/~lzhao9/materials/papers/dlADMM_supp.pdf}}.\\
\textbf{6. Update ${b}^{k+1}_l$}\\
\indent The variables ${b}^{k+1}_l(l=1,\cdots,L)$ are updated as follows:
\begin{align*}
    {b}^{k+1}_l\leftarrow \arg\min\nolimits_{b_l} 
     \!L_\rho({\textbf{W}}^{k+1}_{l},\{\!b^{k+1}_i\!\}_{i\!=\!1}^{l\!-\!1},b_l,\{\!\overline{b}^{k+1}_i\!\}_{i\!=\!l+1}^{L},{\textbf{z}}^{k+1}_{l-1},{\textbf{a}}^{k+1}_{l-1},u^k)
\end{align*}
which is equivalent to the following formulation:
\begin{align*}
        {b}^{k+1}_l\leftarrow \arg\min\nolimits_{b_l} 
     \!\phi({\textbf{W}}^{k+1}_{l},\{\!b^{k+1}_i\!\}_{i\!=\!1}^{l\!-\!1},b_l,\{\!\overline{b}^{k+1}_i\!\}_{i\!=\!l+1}^{L},{\textbf{z}}^{k+1}_{l-1},{\textbf{a}}^{k+1}_{l-1},u^k)
\end{align*}
$U_l(b_l;B)$ is defined as the quadratic approximation of $\phi$ at $\overline{b}^{k+1}_l$ as follows:
\begin{align*}
 U_l(b_l;B)&=\phi(\textbf{W}^{k+1}_{l},\textbf{b}^{k+1}_{l-1},{\textbf{z}}^{k+1}_{l-1},{\textbf{a}}^{k+1}_{l-1},u^k)\\&+\nabla_{\overline{b}^{k+1}_l}\phi^T({\textbf{W}}^{k+1}_{l},{\textbf{b}}^{k+1}_{l-1},{\textbf{z}}^{k+1}_{l-1},{\textbf{a}}^{k+1}_{l-1},u^k)(b_l-\overline{b}^{k+1}_l)\\&+(B/2)\Vert b_l-\overline{b}^{k+1}_l\Vert^2_2.
\end{align*}
where $B> 0$ is a parameter. We set $B=\nu$ for $l=1,\cdots,L-1$ and $B=\rho$ for $l=L$, and solve the resulting subproblems as follows:
\begin{align}
    &{b}^{k+1}_l\leftarrow \arg\min\nolimits_{b_l} {U}_l(b_l;\nu)(l=1,\cdots,L-1) \label{eq:update b}
    \\&{b}^{k+1}_L\leftarrow \arg\min\nolimits_{b_L} {U}_L(b_L;\rho) \label{eq:update bl}
\end{align}.
\indent The solutions to Equations \eqref{eq:update b} and \eqref{eq:update bl} are as follows:
\begin{align*}
    &b^{k+1}_l\leftarrow \overline{b}^{k+1}_l-\nabla_{\overline{b}^{k+1}_L}\phi/\nu (l=1,\cdots,L-1)
    \\&b^{k+1}_L\leftarrow \overline{b}^{k+1}_L-\nabla_{\overline{b}^{k+1}_L}\phi/\rho
\end{align*}
\textbf{7. Update ${z}^{k+1}_l$}\\
\indent The variables ${z}^{k+1}_l(l=1,\cdots,L)$ are updated as follows:
\begin{align*}
    {z}^{k\!+\!1}_{l}&\leftarrow\! \arg\min\nolimits_{z_{l}}\! L_\rho({\textbf{W}}^{k\!+\!1}_{l},{\textbf{b}}^{k\!+\!1}_{l},\{z^{k\!+\!1}_i\}_{i\!=\!1}^{l\!-\!1},z_l,\{\overline{z}^{k\!+\!1}_i\}_{i\!=\!l\!+\!1}^{L},\textbf{a}^{k\!+\!1}_{l\!-\!1},u^k)
\end{align*}
which is equivalent to the following forms for $z_l(l=1,\cdots,L-1)$:
\begin{align}
    {z}^{k\!+\!1}_{l}&\leftarrow\! \arg\min\nolimits_{z_{l}}\! \phi\!({\textbf{W}}^{k\!+\!1}_{l}\!,\!{\textbf{b}}^{k\!+\!1}_{l}\!,\!\{z^{k\!+\!1}_i\}_{i\!=\!1}^{l\!-\!1}\!,\!z_l,\{\overline{z}^{k\!+\!1}_i\}_{i\!=\!l\!+\!1}^{L\!-\!1}\!,\!\textbf{a}^{k\!+\!1}_{l-1}\!,\!u^k)
    \label{eq:update z}
    \end{align}
    and for $z_L$:
\begin{align}
\nonumber
    {z}^{k+1}_L &\leftarrow\arg\min\nolimits_{z_L}      \phi(\textbf{W}^{k+1}_{L},\textbf{b}^{k+1}_{L},\{z^{k+1}_i\}_{i=1}^{L-1},z_L,\textbf{a}^{k}_{L-1},u^k)\\&+R(z_L;y)\label{eq:update zl}
\end{align}
Solving Equations \eqref{eq:update z} and \eqref{eq:update zl} proceeds exactly the same as solving Equations \eqref{eq:update overline z} and \eqref{eq:update overline zl}, respectively.\\
\textbf{8. Update ${a}^{k+1}_l$}\\
\indent The variables ${a}^{k+1}_l(l=1,\cdots,L-1)$ are updated as follows:
\begin{align*}
    {a}^{k+1}_{l}&\leftarrow \arg\min\nolimits_{a_{l}}\! L_\rho({\textbf{W}}^{k\!+\!1}_{l},{\textbf{b}}^{k+1}_{l},{\textbf{z}}^{k\!+\!1}_{l},\{a^k_i\}_{i=1}^{l\!-\!1},a_l,\{\overline{a}^{k\!+\!1}_i\}_{i\!=\!l\!+\!1}^{L\!-\!1},u^k)
\end{align*}
\indent which is equivalent to the following form:
\begin{align*}
{a}^{k\!+\!1}_{l}&\leftarrow\! \arg\min\nolimits_{a_{l}}\! \phi({\textbf{W}}^{k\!+\!1}_{l},{\textbf{b}}^{k\!+\!1}_{l},{\textbf{z}}^{k\!+\!1}_{l},\{a^k_i\}_{i\!=\!1}^{l\!-\!1},a_l,\{\overline{a}^{k\!+\!1}_i\}_{i\!=\!l\!+\!1}^{L\!-\!1},u^k)
\end{align*}
$Q_l(a_l;\tau^{k+1}_l)$ is defined as the quadratic approximation of $\phi$ at $a^{k+1}_l$ as follows:
\begin{align*}{Q}_l(a_{l};{\tau}^{k+1}_l)&=\phi({\textbf{W}}^{k+1}_{l},{\textbf{b}}^{k+1}_{l},{\textbf{z}}^{k+1}_{l},{\textbf{a}}^{k+1}_{l-1},u^k)\\&+(\nabla _{\overline{a}^{k+1}_{l}}\phi)^T({\textbf{W}}^{k+1}_{l},{\textbf{b}}^{k+1}_{l},{\textbf{z}}^{k+1}_{l},{\textbf{a}}^{k+1}_{l-1},u^k)(a_{l}-\overline{a}^{k+1}_{l})\\&+\Vert{\tau}^{k+1}_l\circ (a_{l}-\overline{a}^{k+1}_{l})^{\circ 2}\Vert_{1}/2
\end{align*}
and we can solve the following problem instead:
\begin{align} a^{k+1}_l&\leftarrow\arg\min\nolimits_{a_l} {Q}_l(a_l;{\tau}^{k+1}_l) \label{eq:update a}
\end{align}
where ${\tau}^{k+1}_l>0$ is a  parameter vector. The solution to Equation \eqref{eq:update a} can be obtained by
\begin{align*}
{a}^{k+1}_{l}\leftarrow \overline{a}^{k+1}_{l}-\nabla_{\overline{a}^{k+1}_{l}}\phi/{\tau}^{k+1}_{l}
\end{align*}
To choice of an appropriate  $\tau^{k+1}_l$ is shown  in the supplementary materials\cref{note}.
\section{Convergence Analysis}
\label{sec:convergence}
\indent In this section, the theoretical convergence of the proposed dlADMM algorithm is analyzed. Before we formally present the convergence results of the dlADMM algorithms, Section \ref{sec:assumption} presents necessary assumptions to guarantee the global convgerence of dlADMM. In Section \ref{sec:key properties}, we prove the global convergence of the dlADMM algorithm.
\subsection{Assumptions}
\label{sec:assumption}
\begin{assumption}[Closed-form Solution] There exist activation functions $a_l=f_l(z_l)$ such that Equations \eqref{eq:update overline z} and \eqref{eq:update z} have closed form solutions  $\overline{z}^{k\!+\!1}_l\!=\!\overline{h}(\overline{\textbf{W}}^{k\!+\!1}_{l
+1},\overline{\textbf{b}}^{k\!+\!1}_{l\!+\!1},\overline{\textbf{a}}^{k\!+\!1}_{l})$ and $z^{k\!+\!1}_l\!=\!h(\textbf{W}^{k\!+\!1}_l,\textbf{b}^{k\!+\!1}_l,\textbf{a}^{k\!+\!1}_{l\!-\!1})$, respectively, where $\overline{h}(\bullet)$ and $h(\bullet)$ are continuous functions. \label{ass:assumption 1}
\end{assumption}
\indent This assumption can be satisfied by commonly used activation functions such as ReLU and leaky ReLU. For example, for the ReLU function $a_l=\max(z_l,0)$, Equation \eqref{eq:update z} has the following solution: 
\begin{align*}
    z^{k+1}_l=
    \begin{cases} \min(W^{k+1}_{l}a^{k+1}_{l-1}+b^{k+1}_l,0)&z^{k+1}_l\leq0\\
    \max((W^{k+1}_{l}a^{k+1}_{l-1}+b^{k+1}_l+\overline{a}^{k+1}_l)/2,0)&z^{k+1}_l\geq0\\
    \end{cases}
\end{align*}
\begin{assumption}[Objective Function] $F (\textbf{W}, \textbf{b},\textbf{z}, \textbf{a})$ is coercive over the nonempty set $G = \{(\textbf{W}, \textbf{b},\textbf{z}, \textbf{a}) : z_L - W_La_{L-1} -
b_L = 0\}$. In other words, $F (\textbf{W}, \textbf{b},\textbf{z}, \textbf{a}) \rightarrow \infty$ if $(\textbf{W}, \textbf{b},\textbf{z}, \textbf{a})\in G$ and
$\Vert(\textbf{W}, \textbf{b},\textbf{z},\textbf{ a})\Vert \rightarrow \infty$. Moreover, $R(z_l;y)$ is Lipschitz differentiable with Lipschitz constant $H\geq 0$.\label{ass:assumption 2}
\end{assumption}
The Assumption \ref{ass:assumption 2} is mild enough for most common loss functions to satisfy. For example, the cross-entropy and square loss are Lipschitz differentiable.
\subsection{Key Properties}
\label{sec:key properties}
\indent We present the main convergence result of the proposed dlADMM algorithm in this section. Specifically, as long as Assumptions \ref{ass:assumption 1}-\ref{ass:assumption 2} hold, then Properties \ref{pro:property 1}-\ref{pro:property 3} are satisfied, which are important to prove the global convergence of the proposed dlADMM algorithm. The proof details are included in the supplementary materials\cref{note}.
\begin{property}[Boundness]
\label{pro:property 1}
If $\rho > 2H$, then $\{\textbf{W}^k,\textbf{b}^k
,\textbf{z}^k, \textbf{a}^k,u^k\}$ is bounded, and $L_\rho(\textbf{W}^k,\textbf{b}^k
,\textbf{z}^k, \textbf{a}^k,u^k)$ is lower bounded.
\end{property}
Property \ref{pro:property 1} concludes that all variables and the value of $L_\rho$ have
lower bounds. It is proven under Assumptions \ref{ass:assumption 1} and \ref{ass:assumption 2}, and its proof
can be found in the supplementary materials\cref{note}.
\begin{property}[Sufficient Descent]
\label{pro:property 2}
If $\rho>2H$ so that $C_1=\rho/2-H/2-H^2/\rho>0$, then there exists \\$C_2\!=\!\min(\!\nu/2,C_1,\{\overline{\theta}^{k+1}_l\}_{l=1}^L\!,\!\{{\theta}^{k\!+\!1}_l\}_{l\!=\!1}^L,\{\overline{\tau}^{k\!+\!1}_l\}_{l\!=\!1}^{L\!-\!1},\{\tau^{k\!+\!1}_l\}_{l\!=\!1}^{L\!-\!1})$ such that
\begin{align}
\nonumber &L_\rho(\textbf{W}^k,\textbf{b}^k,\textbf{z}^k,\textbf{a}^k,u^k)-L_\rho(\textbf{W}^{k+1},\textbf{b}^{k+1},\textbf{z}^{k+1},\textbf{a}^{k+1},u^{k+1})\\\nonumber &\geq C_2(\sum\nolimits_{l=1}^L(\Vert \overline{W}_l^{k+1}-W_l^k\Vert^2_2+\Vert {W}_l^{k+1}-\overline{W}_l^{k+1}\Vert^2_2\\\nonumber&+\Vert \overline{b}_l^{k+1}-b_l^k\Vert^2_2+\Vert {b}_l^{k+1}-\overline{b}_l^{k+1}\Vert^2_2)\\&\nonumber +\sum\nolimits_{l=1}^{L-1}(\Vert \overline{a}_l^{k+1}-a_l^k\Vert^2_2+\Vert {a}_l^{k+1}-\overline{a}_l^{k+1}\Vert^2_2)\\&+\Vert \overline{z}^{k+1}_L-{z}^{k}_L\Vert^2_2+\Vert z^{k+1}_L-\overline{z}^{k+1}_L\Vert^2_2)\label{eq: property2} 
\end{align}
\end{property}
\indent Property \ref{pro:property 2} depicts the  monotonic decrease of the objective value during iterations. The proof of Property \ref{pro:property 2} is detailed in the supplementary materials\cref{note}. 
\begin{property}[Subgradient Bound]
\label{pro:property 3}
There exist a constant $C>0$ and $g\in \partial L_\rho(\textbf{W}^{k+1},\textbf{b}^{k+1},\textbf{z}^{k+1},\textbf{a}^{k+1})$ such that
\begin{align}
    \nonumber\Vert g\Vert &\leq C(\Vert \textbf{W}^{k+1}-\overline{\textbf{W}}^{k+1}\Vert+\Vert\textbf{b}^{k+1}-\overline{\textbf{b}}^{k+1}\Vert\\&+\Vert\textbf{z}^{k+1}-\overline{\textbf{z}}^{k+1}\Vert+\Vert\textbf{a}^{k+1}-\overline{\textbf{a}}^{k+1}\Vert+\Vert \textbf{z}^{k+1}-\textbf{z}^k\Vert)\label{eq: property3}
\end{align}
\end{property}
\indent Property \ref{pro:property 3} ensures that the subgradient of the objective function is bounded by variables. The proof of Property \ref{pro:property 3} requires Property \ref{pro:property 1} and the proof is elaborated in the supplementary materials\cref{note}. Now the global convergence of the dlADMM algorithm is presented. The following theorem states that Properties \ref{pro:property 1}-\ref{pro:property 3} are guaranteed.
\begin{theorem}
\label{thero: theorem 1}
For any $\rho>2H$, if Assumptions \ref{ass:assumption 1} and \ref{ass:assumption 2} are satisfied, then Properties \ref{pro:property 1}-\ref{pro:property 3} hold. 
\end{theorem}
\begin{proof}
This theorem can be concluded by the proofs in the supplementary materials\cref{note}.
\end{proof}
\indent The next theorem presents the global convergence of the dlADMM algorithm.
\begin{theorem} [Global Convergence]
\label{thero: theorem 2}
If $\rho>2H$, then for the variables $(\textbf{W},\textbf{b},\textbf{z},\textbf{a},u)$ in Problem \ref{prob:problem 2}, starting from any $(\textbf{W}^0,\textbf{b}^0,\textbf{z}^0,\textbf{a}^0,u^0)$, it has at least a limit point $(\textbf{W}^*,\textbf{b}^*,\textbf{z}^*,\textbf{a}^*,u^*)$, and any limit point $(\textbf{W}^*,\textbf{b}^*,\textbf{z}^*,\textbf{a}^*,u^*)$ is a critical point of Problem \ref{prob:problem 2}. That is, $0\in \partial L_\rho(\textbf{W}^*,\textbf{b}^*,\textbf{z}^*,\textbf{a}^*,u^*)$. Or equivalently, 
\begin{align*}
    & z^*_L=W^*_La^*_{L-1}+b^*_L\\
    & 0\in\partial_{\textbf{W}^*} L_\rho(\textbf{W}^*,\textbf{b}^*,\textbf{z}^*,\textbf{a}^*,u^*)
    & \nabla_{\textbf{b}^*} L_\rho(\textbf{W}^*,\textbf{b}^*,\textbf{z}^*,\textbf{a}^*,u^*)=0\\
    & 0\in\partial_{\textbf{z}^*} L_\rho(\textbf{W}^*,\textbf{b}^*,\textbf{z}^*,\textbf{a}^*,u^*)
    & \nabla_{\textbf{a}^*} L_\rho(\textbf{W}^*,\textbf{b}^*,\textbf{z}^*,\textbf{a}^*,u^*)=0
\end{align*}
\end{theorem}
\begin{proof}
Because $(\textbf{W}^k,\textbf{b}^k,\textbf{z}^k,\textbf{a}^k,u^k)$  is bounded, there exists a subsequence  $(\textbf{W}^s,\textbf{b}^s,\textbf{z}^s,\textbf{a}^s,\textbf{u}^s)$ such that $(\textbf{W}^s,\textbf{b}^s,\textbf{z}^s,\textbf{a}^s,u^s)\rightarrow (\textbf{W}^*,\textbf{b}^*,\textbf{z}^*,\textbf{a}^*,u^*)$ where $(\textbf{W}^*,\textbf{b}^*,\textbf{z}^*,\textbf{a}^*,u^*)$ is a limit point.
By Properties \ref{pro:property 1} and \ref{pro:property 2}, $L_\rho(\textbf{W}^k,\textbf{b}^k,\textbf{z}^k,\textbf{a}^k,u^k)$ is non-increasing and  lower bounded and hence converges.  By Property \ref{pro:property 2}, we prove that $\Vert \overline{\textbf{W}}^{k+1}-\textbf{W}^{k}\Vert\rightarrow 0$, $\Vert \overline{\textbf{b}}^{k+1}-\textbf{b}^{k}\Vert\rightarrow 0$, $\Vert \overline{\textbf{a}}^{k+1}-\textbf{a}^{k}\Vert\rightarrow 0$, $\Vert \textbf{W}^{k+1}-\overline{\textbf{W}}^{k+1}\Vert\rightarrow 0$, $\Vert \textbf{b}^{k+1}-\overline{\textbf{b}}^{k+1}\Vert\rightarrow 0$, and $\Vert \textbf{a}^{k+1}-\overline{\textbf{a}}^{k+1}\Vert\rightarrow 0$, as $k\rightarrow \infty$ . Therefore $\Vert \textbf{W}^{k+1}-{\textbf{W}}^{k}\Vert\rightarrow 0$, $\Vert \textbf{b}^{k+1}-{\textbf{b}}^{k}\Vert\rightarrow 0$, and $\Vert \textbf{a}^{k+1}-{\textbf{a}}^{k}\Vert\rightarrow 0$, as $k\rightarrow \infty$. Moreover, from Assumption \ref{ass:assumption 1}, we know that $\overline{\textbf{z}}^{k+1}\rightarrow \textbf{z}^{k}$ and $\textbf{z}^{k+1}\rightarrow \overline{\textbf{z}}^{k+1}$  as $k\rightarrow \infty$. Therefore, $\textbf{z}^{k+1}\rightarrow \textbf{z}^{k}$.  We infer there exists $g^k\in \partial L_\rho(\textbf{W}^k,\textbf{b}^k,\textbf{z}^k,\textbf{a}^k,u^k)$ such that $\Vert g^k\Vert \rightarrow 0$ as $k\rightarrow \infty$ based on Property \ref{pro:property 3}. Specifically, $\Vert g^s\Vert \rightarrow 0$ as $s\rightarrow \infty$. According to the definition of general subgradient (Defintion 8.3 in \cite{rockafellar2009variational}), we have $0\in \partial L_\rho(\textbf{W}^*,\textbf{b}^*,\textbf{z}^*,\textbf{a}^*,u^*)$. In other words, the limit point $(\textbf{W}^*,\textbf{b}^*,\textbf{z}^*,\textbf{a}^*,u^*)$ is a critical point of $L_\rho$ defined in Equation \eqref{eq: Lagrangian}.
\end{proof}
\indent Theorem \ref{thero: theorem 2} shows that our dlADMM algorithm converges globally for sufficiently large $\rho$, which is consistent with previous literature \cite{wang2015global,kiaee2016alternating}. The next theorem shows that the dlADMM converges globally with a sublinear convergence rate $o(1/k)$.
\begin{theorem}[Convergence Rate]
For a sequence\\ $(\textbf{W}^k,\textbf{b}^k,\textbf{z}^k,\textbf{a}^k,u^k)$,  define $c_k=\min\nolimits_{0\leq i\leq k}(\sum\nolimits_{l=1}^L(\Vert \overline{W}_l^{i+1}-W_l^i\Vert^2_2+\Vert {W}_l^{i+1}-\overline{W}_l^{i+1}\Vert^2_2+\Vert \overline{b}_l^{i+1}-b_l^i\Vert^2_2+\Vert {b}_l^{i+1}-\overline{b}_l^{i+1}\Vert^2_2) +\sum\nolimits_{l=1}^{L-1}(\Vert \overline{a}_l^{i+1}-a_l^i\Vert^2_2+\Vert {a}_l^{i+1}-\overline{a}_l^{i+1}\Vert^2_2)+\Vert \overline{z}^{i+1}_L-{z}^{i}_L\Vert^2_2+\Vert z^{i+1}_L-\overline{z}^{i+1}_L\Vert^2_2)$, then the convergence rate of $c_k$ is $o(1/k)$.
\label{thero: theorem 3}
\end{theorem}
\begin{proof}
The proof of this theorem is included in the supplementary materials\cref{note}. 
\end{proof}
\section{Experiments}
\label{sec:experiment}
In this section, we evaluate dlADMM algorithm using benchmark datasets. Effectiveness, efficiency and convergence properties of dlADMM are compared with state-of-the-art methods. All experiments were conducted on 64-bit Ubuntu16.04 LTS with Intel(R) Xeon processor and GTX1080Ti GPU. 

\subsection{Experiment Setup}
\subsubsection{Dataset} In this experiment, two benchmark datasets were used for performance evaluation:  MNIST \cite{lecun1998gradient} and  Fashion MNIST \cite{xiao2017fashion}. The MNIST dataset has ten classes of handwritten-digit images, which was firstly introduced by Lecun et al. in 1998 \cite{lecun1998gradient}. It contains 55,000 training samples and 10,000 test samples with 784 features each, which is provided by the Keras library \cite{chollet2017deep}. Unlike the MNIST dataset, the Fashion MNIST dataset has ten classes of assortment images on the website of Zalando, which is Europe's largest online fashion platform \cite{xiao2017fashion}. The Fashion-MNIST dataset consists of 60,000 training samples and 10,000 test samples with 784 features each.
\subsubsection{Experiment Settings}
\indent We set up a network architecture which contained two hidden layers with $1,000$ hidden units each. The Rectified linear unit (ReLU) was used for the activation function for both network structures. The loss function was set as the deterministic cross-entropy loss. $\nu$ was set to $10^{-6}$. $\rho$ was initialized as $10^{-6}$ and was multiplied by 10 every 100 iterations. The number of iteration was set to $200$. In the experiment, one iteration means one epoch.
\subsubsection{Comparison Methods}
\indent Since this paper focuses on fully-connected deep neural networks, SGD and its variants and ADMM are state-of-the-art methods and hence were served as comparison methods. For SGD-based methods, the full batch dataset is used for training models. All parameters were chosen by the accuracy of the training dataset. The baselines are described as follows: \\
\indent 1. Stochastic Gradient Descent (SGD) \cite{bottou2010large}. The SGD and its variants are the most popular deep learning optimizers, whose convergence has been studied extensively in the literature. The learning rate of SGD was set to $10^{-6}$ for both the MNIST and Fashion MNIST datasets.\\
\indent 2. Adaptive gradient algorithm (Adagrad) \cite{duchi2011adaptive}. Adagrad is an improved version of SGD: rather than fixing the learning rate during iteration, it adapts the learning rate to the hyperparameter. The learning rate of Adagrad was set to $10^{-3}$ for both the MNIST and Fashion MNIST datasets.\\
\indent 3. Adaptive learning rate method (Adadelta) \cite{zeiler2012adadelta}. As an improved version of the Adagrad, the Adadelta is proposed to overcome the sensitivity to hyperparameter selection. The learning rate of Adadelta was set to $0.1$ for both the MNIST and Fashion MNIST datasets.\\
\indent 4. Adaptive momentum estimation (Adam) \cite{kingma2014adam}. Adam is the most popular optimization method for deep learning models. It estimates the first and second momentum in order to correct the biased gradient and thus makes convergence fast. The learning rate of Adam was set to $10^{-3}$ for both the MNIST and Fashion MNIST datasets.\\
\indent 5. Alternating Direction Method of Multipliers (ADMM) \cite{taylor2016training}. ADMM is a powerful convex optimization method because it can split an objective function into a series of subproblems, which are coordinated to get global solutions. It is scalable to large-scale datasets and supports parallel computations. The $\rho$ of ADMM was set to $1$ for both the MNIST and Fashion MNIST datasets.\\
\subsection{Experimental Results}
\indent In this section, experimental results of the dlADMM algorithm are analyzed against comparison methods.
\subsubsection{Convergence}
\begin{figure}[h]
  \centering
\small
\begin{minipage}
{0.49\linewidth}
\centerline{\includegraphics[width=\columnwidth]
{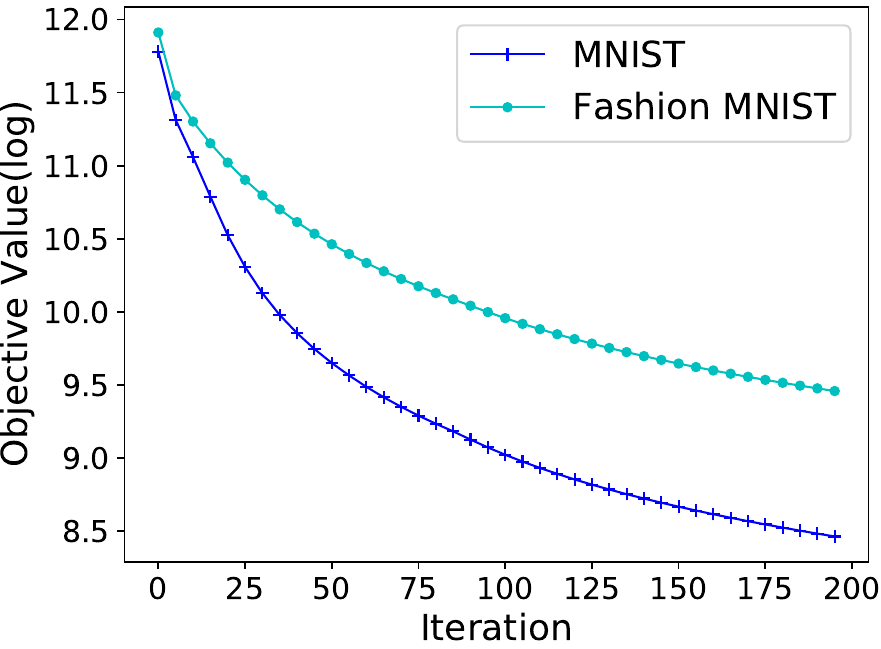}}
\centerline{(a). Objective value}
\end{minipage}
\hfill
\begin{minipage}
{0.49\linewidth}
\centerline{\includegraphics[width=\columnwidth]
{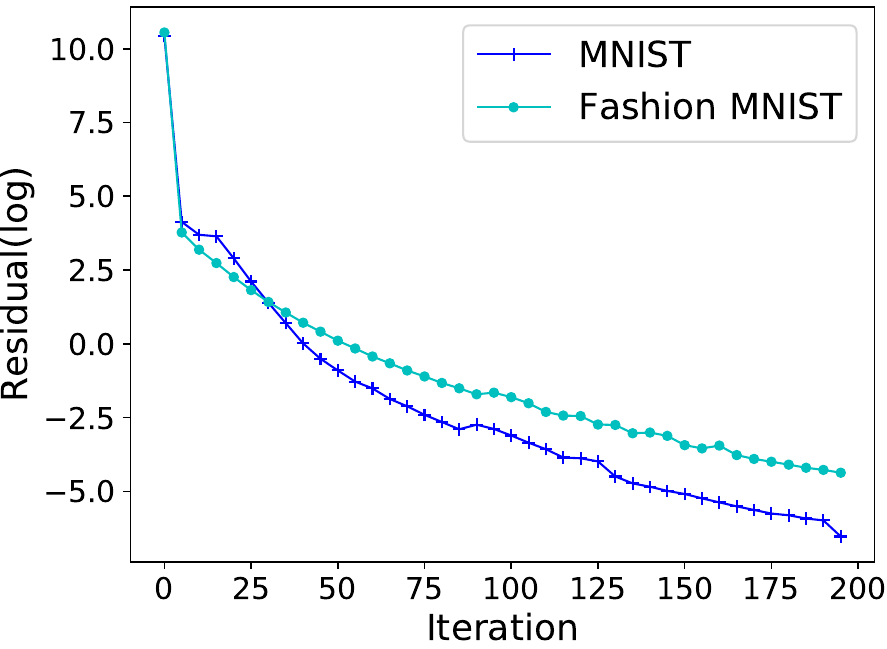}}
\centerline{(b).  Residual}
\end{minipage}
  \caption{Convergence curves of dlADMM algorithm for MNIST and Fashion MNIST datasets when $\rho=1$: dlADMM algorithm converged.}
  \label{fig:convergence}
\end{figure}
\begin{figure}[h]
  \centering
\begin{minipage}
{0.49\linewidth}
\centerline{\includegraphics[width=\textwidth]
{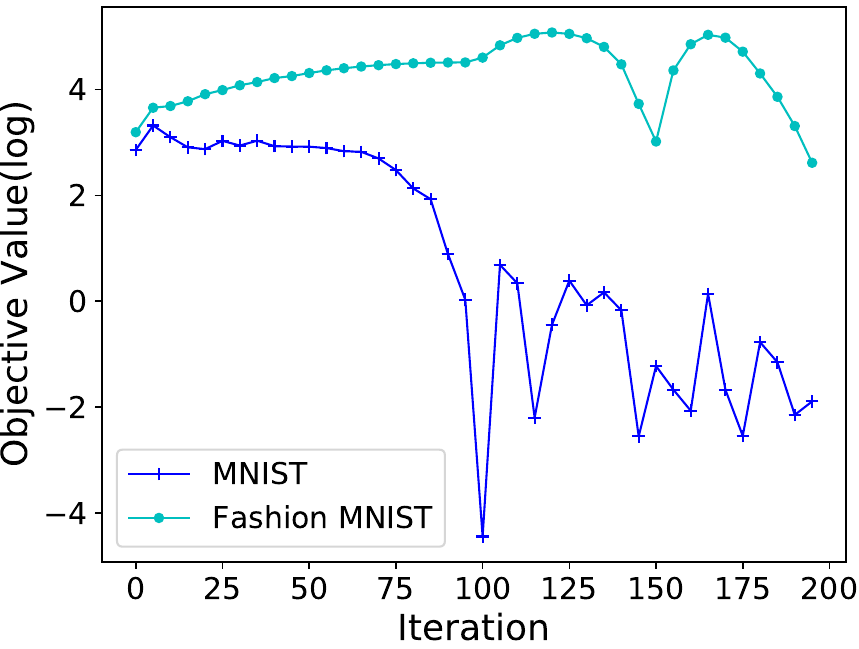}}
\centerline{(a). Objective value}
\end{minipage}
\hfill
\begin{minipage}
{0.49\linewidth}
\centerline{\includegraphics[width=\textwidth]
{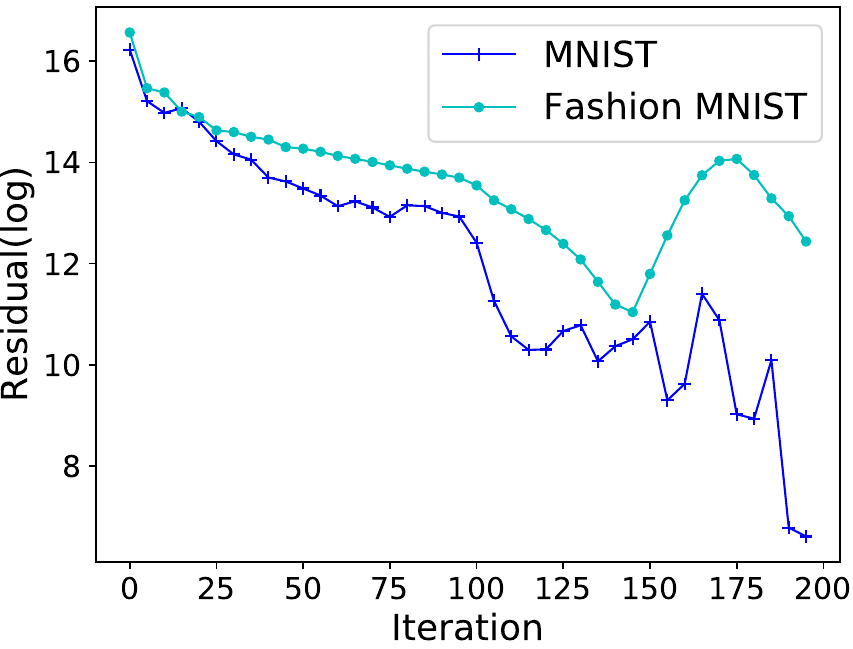}}
\centerline{(b). Residual}
\end{minipage}
  \caption{Divergence curves of the dlADMM algorithm for the MNIST and the Fashion MNIST datasets when $\rho=10^{-6}$: dlADMM algorithm diverged.}
  \label{fig:divergence}
\end{figure}
\indent First, we show that our proposed dlADMM algorithm converges when $\rho$ is sufficiently large and diverges when $\rho$ is small for both the MNIST dataset and the Fashion MNIST dataset.\\
\indent
The convergence and divergence of dlADMM algorithm are shown in Figures \ref{fig:convergence} and \ref{fig:divergence} when $\rho=1$ and $\rho=10^{-6}$ ,respectively. In Figures \ref{fig:convergence}(a) and \ref{fig:divergence}(a), the X axis and Y axis denote the number of iterations and the logarithm of objective value, respectively. In Figures, \ref{fig:convergence}(b) and \ref{fig:divergence}(b), the X axis and Y axis denote the number of iterations and the logarithm of the residual, respectively. Figure \ref{fig:convergence}, both the objective value and the residual decreased monotonically for the MNIST dataset and the Fashion-MNIST dataset, which validates our theoretical guarantees in Theorem \ref{thero: theorem 2}. Moreover, Figure \ref{fig:divergence} illustrates that  both the objective value and  the residual diverge when $\rho=10^{-6}$. The curves fluctuated drastically on the objective value. Even though there was a decreasing trend for the residual, it still fluctuated irregularly and failed to converge.
\subsubsection{Performance}
\begin{figure}[h]
  \centering
\begin{minipage}
{0.49\linewidth}
\centerline{\includegraphics[width=\textwidth]
{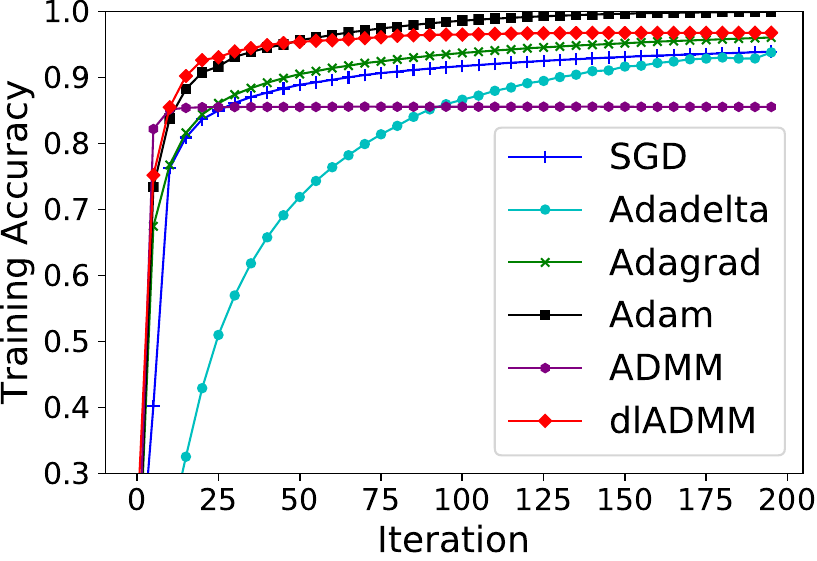}}
\centerline{(a). Training accuracy}
\end{minipage}
\hfill
\begin{minipage}
{0.49\linewidth}
\centerline{\includegraphics[width=\textwidth]
{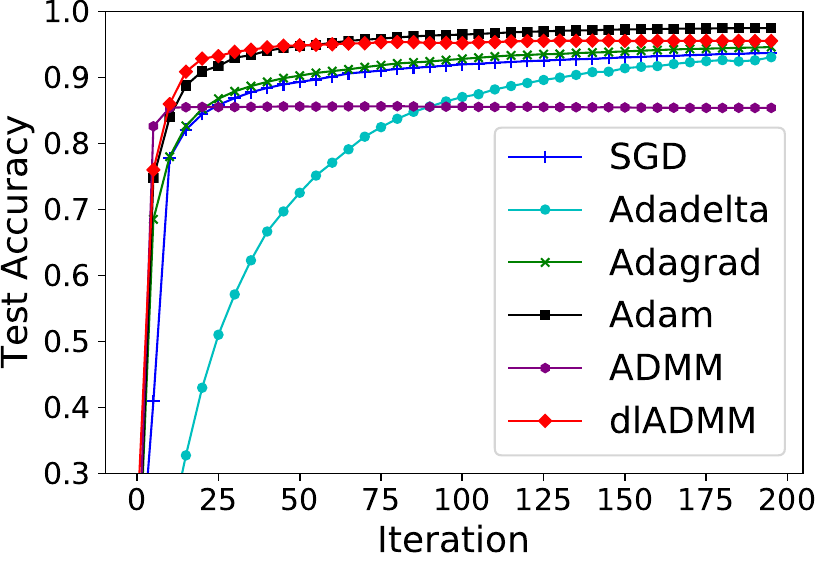}}
\centerline{(b). Test accuracy}
\end{minipage}
  \caption{Performance of all methods for the MNIST dataset: dlADMM algorithm outperformed most of the comparison methods.}
  \label{fig:MNIST}
\end{figure}
\begin{figure}[h]
  \centering
\begin{minipage}
{0.49\linewidth}
\centerline{\includegraphics[width=\textwidth]
{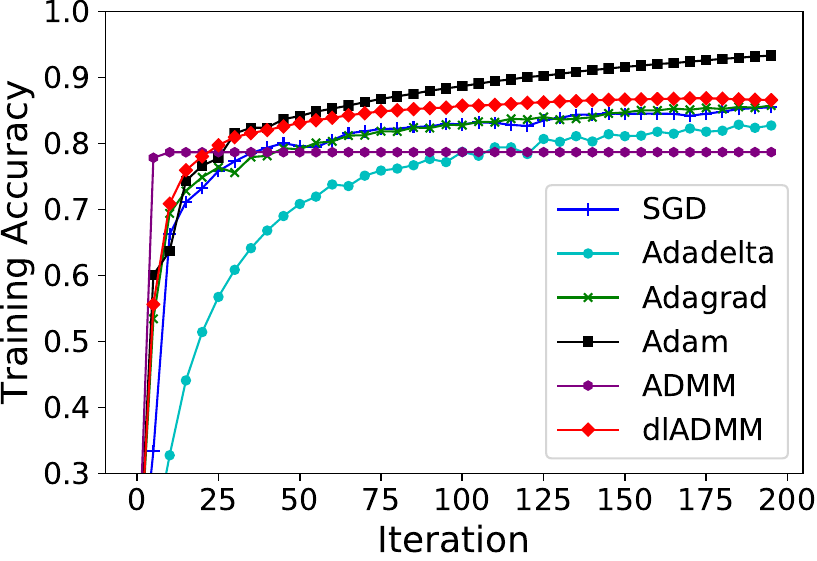}}
\centerline{(a). Training Accuracy}
\end{minipage}
\hfill
\begin{minipage}
{0.49\linewidth}
\centerline{\includegraphics[width=\textwidth]
{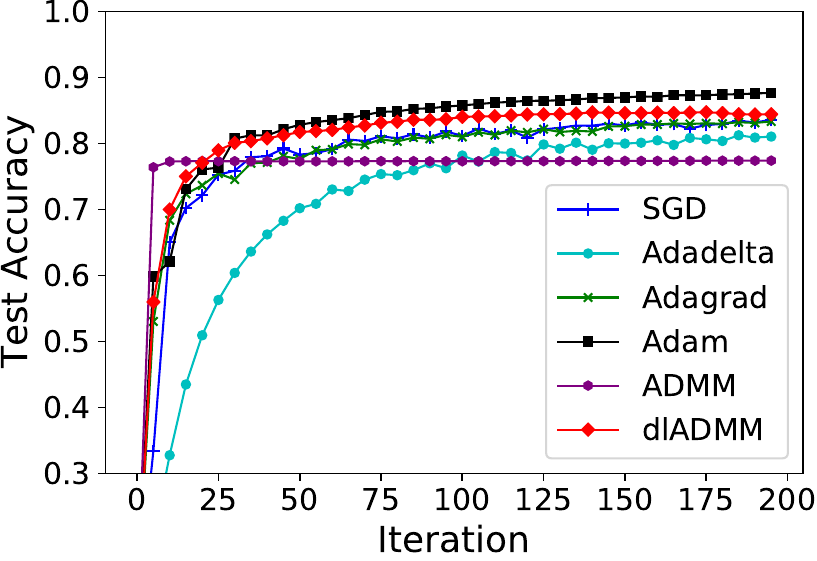}}
\centerline{(b).Test Accuracy}
\end{minipage}
  \caption{Performance of all methods for the Fashion MNIST dataset: dlADMM algorithm outperformed most of the comparsion methods.}
  \label{fig:FMNIST}
\end{figure}
\begin{table}[!hbp]
\scriptsize
\centering
\begin{tabular}{c|c|c|c|c|c}
\hline\hline
\multicolumn{6}{c}{MNIST dataset: From 200 to 1,000 neurons}\\
\hline
\diagbox{$\rho$}{neurons} &200&400 &600  &800 &1000 \\\hline
$10^{-6}$&1.9025&2.7750&3.6615&4.5709&5.7988\\\hline
 $10^{-5}$&2.8778&4.6197&6.3620&8.2563&10.0323\\\hline
 $10^{-4}$&2.2761&3.9745&5.8645&7.6656&9.9221\\\hline
 $10^{-3}$&2.4361&4.3284&6.5651&8.7357&11.3736\\\hline
 $10^{-2}$&2.7912&5.1383&7.8249&10.0300&13.4485\\
\hline\hline
\multicolumn{6}{c}{Fashion MNIST dataset: From 200 to 1,000 neurons }\\
\hline
\diagbox{$\rho$}{neurons} &200&400&600&800&1000\\
\hline
$10^{-6}$&2.0069&2.8694&4.0506&5.1438&6.7406\\\hline
$10^{-5}$&3.3445&5.4190&7.3785&9.0813&11.0531\\\hline
$10^{-4}$&2.4974&4.3729&6.4257&8.3520&10.0728\\\hline
$10^{-3}$&2.7108&4.7236&7.1507&9.4534&12.3326\\\hline
$10^{-2}$&2.9577&5.4173&8.2518&10.0945&14.3465
\\\hline\hline
\end{tabular}
\caption{The relationship between running time per iteration (in second) and the number of neurons for each layer as well as value of $\rho$ when the training size was fixed: generally, the running time increased as the number of neurons and the value of $\rho$ became larger.}
\label{tab:running time 1}
\end{table}
\begin{table}[!hbp]
\centering
\scriptsize
\begin{tabular}{c|c|c|c|c|c}
\hline\hline
\multicolumn{6}{c}{MNIST dataset: From 11,000  to 55,000 training samples}\\
\hline
\diagbox{$\rho$}{size} &11,000&22,000 &33,000  &44,000 &55,000 \\\hline
$10^{-6}$&1.0670&2.0682&3.3089&4.6546&5.7709\\\hline
$10^{-5}$&2.3981&3.9086&6.2175&7.9188&10.2741\\\hline
 $10^{-4}$&2.1290&3.7891&5.6843&7.7625&9.8843\\\hline
 $10^{-3}$&2.1295&4.1939&6.5039&8.8835&11.3368\\\hline
 $10^{-2}$&2.5154&4.9638&7.6606&10.4580&13.4021\\
\hline\hline
\multicolumn{6}{c}{Fashion MNIST dataset: From 12,000  to 60,000 training samples }\\
\hline
\diagbox{$\rho$}{size} &12,000&24,000&36,000&48,000&60,000\\
\hline
$10^{-6}$&1.2163&2.3376&3.7053&5.1491&6.7298\\\hline
$10^{-5}$&2.5772&4.3417&6.6681&8.3763&11.0292\\\hline
$10^{-4}$&2.3216&4.1163&6.2355&8.3819&10.7120\\\hline
$10^{-3}$&2.3149&4.5250&6.9834&9.5853&12.3232\\\hline
$10^{-2}$&2.7381&5.3373&8.1585&11.1992&14.2487
\\\hline\hline
\end{tabular}
\caption{The relationship between running time per iteration (in second) and the size of  training samples as well as value of $\rho$ when the number of neurons is fixed: generally, the running time increased as the training sample and the value of $\rho$ became larger.}
\label{tab:running time 2}
\end{table}
\indent Figure \ref{fig:MNIST} and Figure \ref{fig:FMNIST} show the curves of the training accuracy and test accuracy of  our proposed dlADMM algorithm and baselines, respectively. Overall, both the training accuracy and the test accuracy of  our proposed dlADMM algorithm outperformed most baselines for both the MNIST dataset and the Fashion MNIST dataset.  Specifically, the curves of our dlADMM algorithm soared to $0.8$ at the early stage, and then raised steadily towards to $0.9$ or more. The curves of the most SGD-related methods, SGD, Adadelta, and Adagrad, moved more slowly than our proposed dlADMM algorithm. The curves of the ADMM also rocketed to around $0.8$, but decreased slightly later on. Only the state-of-the-art Adam performed better than dlADMM slightly.

\subsubsection{Scalability Analysis}
\indent In this subsection, the relationship between running time per iteration of our proposed dlADMM algorithm and three potential factors, namely, the value of $\rho$, the size of training samples, and the number of neurons was explored. The running time was calculated by the average of 200 iterations. \\
\indent Firstly, when the training size was fixed, the computational result for the MNIST dataset and Fashion MNIST dataset is shown in Table \ref{tab:running time 1}. The number of neurons for each layer ranged from 200 to 1,000, with an increase of 200 each time. The value of $\rho$ ranged from $10^{-6}$ to $10^{-2}$, with multiplying by 10  each time. Generally, the running time increased as the number of neurons and the value of $\rho$ became larger. However, there were a few exceptions: for example, when there were 200 neurons for the MNIST dataset, and $\rho$ increased from $10^{-5}$ to $10^{-4}$, the running time per iteration dropped from  $2.8778$ seconds to $2.2761$ seconds.\\
\indent Secondly, we fixed the number of neurons for each layer as $1,000$. The relationship between running time per iteration, the training size and the value of $\rho$ is shown in Table \ref{tab:running time 2}. The value of $\rho$ ranged from $10^{-6}$ to $10^{-2}$, with multiplying by 10  each time. The training size of the MNIST dataset ranged from $11,000$ to $55,000$, with an increase of $11,000$ each time. The training size of the Fashion MNIST dataset ranged from $12,000$ to $60,000$, with an increase of $12,000$ each time. Similiar to Table \ref{tab:running time 2}, the running time increased generally as the training sample and the value of $\rho$ became larger and some exceptions exist.
\section{Conclusion and Future Work}
\label{sec:conclusion}
\indent Alternating Direction Method of Multipliers (ADMM) is a good alternative to Stochastic Gradient Descent (SGD) for deep learning problems. In this paper, we propose a novel deep learning Alternating Direction Method of Multipliers (dlADMM) to address some previously mentioned challenges. Firstly, the dlADMM  updates parameters from backward to forward in order to transmit parameter information more efficiently. The time complexity is successfully reduced from $O(n^3)$ to $O(n^2)$  by iterative quadratic approximations and backtracking. Finally, the dlADMM is guaranteed to converge to a critical solution under mild conditions. Experiments on benchmark datasets demonstrate that our proposed dlADMM algorithm outperformed most of the comparison methods.\\
\indent In the future, we may extend our dlADMM from the fully-connected neural network to the famous Convolutional Neural Network (CNN) or Recurrent Neural Network (RNN), because our convergence guarantee is also applied to them.  We also consider other nonlinear activation functions such as sigmoid and hyperbolic tangent function (tanh).
\section*{Acknowledgement}
\indent This work was supported by the National Science Foundation grant: $\#1755850$.
\bibliography{example_paper}

\begin{thebibliography}{10}

\bibitem{beck2009fast}
Amir Beck and Marc Teboulle.
\newblock A fast iterative shrinkage-thresholding algorithm for linear inverse
  problems.
\newblock {\em SIAM journal on imaging sciences}, 2(1):183--202, 2009.

\bibitem{bottou2010large}
L{\'e}on Bottou.
\newblock Large-scale machine learning with stochastic gradient descent.
\newblock In {\em Proceedings of COMPSTAT'2010}, pages 177--186. Springer,
  2010.

\bibitem{boyd2011distributed}
Stephen Boyd, Neal Parikh, Eric Chu, Borja Peleato, and Jonathan Eckstein.
\newblock Distributed optimization and statistical learning via the alternating
  direction method of multipliers.
\newblock {\em Foundations and Trends{\textregistered} in Machine Learning},
  3(1):1--122, 2011.

\bibitem{chen2015extended}
Caihua Chen, Min Li, Xin Liu, and Yinyu Ye.
\newblock Extended admm and bcd for nonseparable convex minimization models
  with quadratic coupling terms: convergence analysis and insights.
\newblock {\em Mathematical Programming}, pages 1--41, 2015.

\bibitem{chollet2017deep}
Francois Chollet.
\newblock {\em Deep learning with python}.
\newblock Manning Publications Co., 2017.

\bibitem{deng2017parallel}
Wei Deng, Ming-Jun Lai, Zhimin Peng, and Wotao Yin.
\newblock Parallel multi-block admm with o (1/k) convergence.
\newblock {\em Journal of Scientific Computing}, 71(2):712--736, 2017.

\bibitem{duchi2011adaptive}
John Duchi, Elad Hazan, and Yoram Singer.
\newblock Adaptive subgradient methods for online learning and stochastic
  optimization.
\newblock {\em Journal of Machine Learning Research}, 12(Jul):2121--2159, 2011.

\bibitem{gao2019incomplete}
Yuyang Gao, Liang Zhao, Lingfei Wu, Yanfang Ye, Hui Xiong, and Chaowei Yang.
\newblock Incomplete label multi-task deep learning for spatio-temporal event
  subtype forecasting.
\newblock 2019.

\bibitem{kiaee2016alternating}
Farkhondeh Kiaee, Christian Gagn{\'e}, and Mahdieh Abbasi.
\newblock Alternating direction method of multipliers for sparse convolutional
  neural networks.
\newblock {\em arXiv preprint arXiv:1611.01590}, 2016.

\bibitem{kingma2014adam}
Diederik~P Kingma and Jimmy Ba.
\newblock Adam: A method for stochastic optimization.
\newblock {\em arXiv preprint arXiv:1412.6980}, 2014.

\bibitem{krizhevsky2012imagenet}
Alex Krizhevsky, Ilya Sutskever, and Geoffrey~E Hinton.
\newblock Imagenet classification with deep convolutional neural networks.
\newblock In {\em Advances in neural information processing systems}, pages
  1097--1105, 2012.

\bibitem{lecun2015deep}
Yann LeCun, Yoshua Bengio, and Geoffrey Hinton.
\newblock Deep learning.
\newblock {\em nature}, 521(7553):436, 2015.

\bibitem{lecun1998gradient}
Yann LeCun, L{\'e}on Bottou, Yoshua Bengio, and Patrick Haffner.
\newblock Gradient-based learning applied to document recognition.
\newblock {\em Proceedings of the IEEE}, 86(11):2278--2324, 1998.

\bibitem{mikolov2010recurrent}
Tom{\'a}{\v{s}} Mikolov, Martin Karafi{\'a}t, Luk{\'a}{\v{s}} Burget, Jan
  {\v{C}}ernock{\`y}, and Sanjeev Khudanpur.
\newblock Recurrent neural network based language model.
\newblock In {\em Eleventh Annual Conference of the International Speech
  Communication Association}, 2010.

\bibitem{polyak1964some}
Boris~T Polyak.
\newblock Some methods of speeding up the convergence of iteration methods.
\newblock {\em USSR Computational Mathematics and Mathematical Physics},
  4(5):1--17, 1964.

\bibitem{j.2018on}
Sashank~J. Reddi, Satyen Kale, and Sanjiv Kumar.
\newblock On the convergence of adam and beyond.
\newblock In {\em International Conference on Learning Representations}, 2018.

\bibitem{rockafellar2009variational}
R~Tyrrell Rockafellar and Roger J-B Wets.
\newblock {\em Variational analysis}, volume 317.
\newblock Springer Science \& Business Media, 2009.

\bibitem{rumelhart1986learning}
David~E Rumelhart, Geoffrey~E Hinton, and Ronald~J Williams.
\newblock Learning representations by back-propagating errors.
\newblock {\em nature}, 323(6088):533, 1986.

\bibitem{sutskever2013importance}
Ilya Sutskever, James Martens, George Dahl, and Geoffrey Hinton.
\newblock On the importance of initialization and momentum in deep learning.
\newblock In {\em International conference on machine learning}, pages
  1139--1147, 2013.

\bibitem{taylor2016training}
Gavin Taylor, Ryan Burmeister, Zheng Xu, Bharat Singh, Ankit Patel, and Tom
  Goldstein.
\newblock Training neural networks without gradients: A scalable admm approach.
\newblock In {\em International Conference on Machine Learning}, pages
  2722--2731, 2016.

\bibitem{tielemandivide}
T~Tieleman and G~Hinton.
\newblock Divide the gradient by a running average of its recent magnitude.
  coursera: Neural networks for machine learning.
\newblock Technical report, Technical Report. Available online: https://zh.
  coursera.
  org/learn/neuralnetworks/lecture/YQHki/rmsprop-divide-the-gradient-by-a-running-average-of-its-recent-magnitude
  (accessed on 21 April 2017).

\bibitem{wang2017nonconvex}
Junxiang Wang and Liang Zhao.
\newblock Nonconvex generalizations of admm for nonlinear equality constrained
  problems.
\newblock {\em arXiv preprint arXiv:1705.03412}, 2017.

\bibitem{wang2019multi}
Junxiang Wang, Liang Zhao, and Lingfei Wu.
\newblock Multi-convex inequality-constrained alternating direction method of
  multipliers.
\newblock {\em arXiv preprint arXiv:1902.10882}, 2019.

\bibitem{wang2015global}
Yu~Wang, Wotao Yin, and Jinshan Zeng.
\newblock Global convergence of admm in nonconvex nonsmooth optimization.
\newblock {\em Journal of Scientific Computing}, pages 1--35, 2015.

\bibitem{xiao2017fashion}
Han Xiao, Kashif Rasul, and Roland Vollgraf.
\newblock Fashion-mnist: a novel image dataset for benchmarking machine
  learning algorithms.
\newblock {\em arXiv preprint arXiv:1708.07747}, 2017.

\bibitem{zeiler2012adadelta}
Matthew~D Zeiler.
\newblock Adadelta: an adaptive learning rate method.
\newblock {\em arXiv preprint arXiv:1212.5701}, 2012.

\end{thebibliography}
\bibliographystyle{plain}
\newpage
\onecolumn
\large{Supplementary Materials}
\begin{appendix}
\small
\section{Algorithms to update $W^{k+1}_l$ and $a^{k+1}_l$}
\label{sec:same algorithm}
\indent The algorithms to update $W^{k+1}_l$ and $a^{k+1}_l$ are described in the Algorithms \ref{algo:theta update} and \ref{algo:tau update}, respectively.
\begin{algorithm} 
\caption{The Backtracking Algorithm  to update ${W}^{k+1}_{l}$ } 
\begin{algorithmic}[1]
\label{algo:theta update}
\REQUIRE ${\textbf{W}}^{k+1}_{l-1}$,${\textbf{b}}^{k+1}_{l-1}$, ${\textbf{z}}^{k+1}_{l-1}$,${\textbf{a}}^{k+1}_{l-1}$,$u^k$, $\rho$, some constant ${\gamma}>1$. 
\ENSURE ${\theta}^{k+1}_l$,${{W}}^{k+1}_{l}$. 
\STATE Pick up ${\alpha}$ and ${\zeta}=\overline{W}^{k+1}_l-\nabla_{\overline{W}^{k+1}_l}\phi/{\alpha}$.
\WHILE{$\phi(\{W^{k+1}_i\}_{i=1}^{l-1},{\zeta},\{\overline{W}^{k+1}_i\}_{i=l+1}^{L},{\textbf{b}}^{k+1}_{l},{\textbf{z}}^{k+1}_{l},{\textbf{a}}^{k+1}_{l},u^k)>{P}_l({\zeta};{\alpha})$}
\STATE ${\alpha}\leftarrow {\alpha}\ {\gamma}$.\\
\STATE Solve ${\zeta}$ by Equation \eqref{eq:update W}.\\
\ENDWHILE
\STATE Output ${\theta}^{k+1}_l \leftarrow {\alpha} $.\\
\STATE Output ${W}^{k+1}_{l}\leftarrow {\zeta}$.
\end{algorithmic}
\end{algorithm}
\begin{algorithm} 
\caption{The Backtracking Algorithm  to update ${a}^{k+1}_{l}$ } 
\begin{algorithmic}[1]
\label{algo:tau update}
\REQUIRE ${\textbf{W}}^{k+1}_{l}$,${\textbf{b}}^{k+1}_{l}$, ${\textbf{z}}^{k+1}_{l}$,${\textbf{a}}^{k+1}_{l-1}$,$u^k$, $\rho$, some constant ${\eta}>1$. 
\ENSURE ${\tau}^{k+1}_l$,${{a}}^{k+1}_{l}$. 
\STATE Pick up ${t}$ and ${\beta}=\overline{a}^{k+1}_l-\nabla_{\overline{a}^{k+1}_l}\phi/{t}$
\WHILE{$\phi({\textbf{W}}^{k+1}_{l+1},{\textbf{b}}^{k+1}_{l+1},{\textbf{z}}^{k+1}_{l+1},\{a^{k+1}_i\}_{i=1}^{l-1},{\beta},\{\overline{a}^{k+1}_i\}_{i=l+1}^{L-1},u^k)>{Q}_l({\beta};{t})$}
\STATE ${t}\leftarrow {t}{\eta}$.\\
\STATE ${\beta}\leftarrow \overline{a}^{k+1}_{l}-\nabla_{\overline{a}^{k+1}_{l}}\phi/{t}$.\\
\ENDWHILE
\STATE Output ${\tau}^{k+1}_l \leftarrow {t} $.\\
\STATE Output ${a}^{k+1}_{l}\leftarrow {\beta}$.
\end{algorithmic}
\end{algorithm}
\section{Lemmas for the Proofs of Properties}
\label{sec:proofs}
The following several lemmas are preliminary results.
\begin{lemma}
\label{lemma:lemma 1}
Equation \eqref{eq:update overline W} holds if and only if there exists $\overline{s}\in \partial\Omega_l(\overline{W}^{k+1}_l)$, the subgradient of $\Omega_l(\overline{W}^{k+1}_l)$  such that
\begin{align*}
    &\nabla_{{W}^{k}_l}\phi({\overline{\textbf{W}}}^{k+1}_{l+1},\overline{\textbf{b}}^{k+1}_{l},\overline{\textbf{z}}^{k+1}_{l},\overline{\textbf{a}}^{k+1}_{l},u^k) +\overline{\theta}^{k+1}_l \circ(\overline{W}^{k+1}_l-{W}^{k}_l)+\overline{s}=0
\end{align*}
Likewise, Equation \eqref{eq:update W} holds if and only if there exists $s\in \partial\Omega_l(W^{k+1}_l)$, the subgradient of $\Omega_l(W^{k+1}_l)$  such that
\begin{align*}
    &\nabla_{\overline{W}^{k+1}_l}\phi({\textbf{W}}^{k+1}_{l-1},\textbf{b}^{k+1}_{l-1},\textbf{z}^{k+1}_{l-1},{\textbf{a}}^{k+1}_{l-1},u^k) +\theta^{k+1}_l \circ(W^{k+1}_l-\overline{W}^{k+1}_l)+s=0
\end{align*}
\end{lemma}
\begin{proof}
    These can be obtained by directly applying the optimality conditions of Equation \eqref{eq:update overline W} and Equation \eqref{eq:update W}, respectively.
\end{proof}
\begin{lemma}
$\nabla_{z_L^k} R(z_L^k;y)+u^k=0$ for all $k\in \mathbb{N}$.
\label{lemma:z_l optimality}
\end{lemma}
\begin{proof}
The optimality condition of $z^k_L$ in Equation \eqref{eq:update zl} gives rise to 
\begin{align*}
    \nabla_{z^k_L} R(z^{k}_L;y)+\rho(z^{k}_L-W^{k}_La^{k}_{L-1}-b^{k}_L)+u^{k-1}=0
\end{align*}
Because $u^k=u^{k-1}+\rho(z^{k}_L-W^{k}_La^{k}_{L-1}-b^{k}_L)$, then we have $\nabla_{z_L^k} R(z_L^k;y)+u^k=0$.
\end{proof}
\begin{lemma}
\label{lemma: R(z_l) lipschitz}
It holds that $\forall z_{L,1},z_{L,2}\in\mathbb{R}^{n_L}$,
\begin{align*}
    &R(z_{L,1};y)\leq R(z_{L,2};y)+\nabla_{z_{L,2}} R^T(z_{L,2};y)(z_{L,1}-z_{L,2})+(H/2)\Vert z_{L,1}-z_{L,2}\Vert^2\\&-R(z_{L,1};y)\leq -R(z_{L,2};y)-\nabla_{z_{L,2}} R^T(z_{L,2};y)(z_{L,1}-z_{L,2})+(H/2)\Vert z_{L,1}-z_{L,2}\Vert^2
\end{align*}
\end{lemma}
\begin{proof}
Because $R(z_L;y)$ is Lipschitz differentiable by Assumption \ref{ass:assumption 2}, so is $-R(z_L;y)$. Therefore, this lemma is proven exactly as same as Lemma 2.1 in \cite{beck2009fast}.
\end{proof}
\begin{lemma}
For Equations \eqref{eq:update overline b} and \eqref{eq:update b}, if $\overline{B},B\geq \nu$,then the following inequalities hold:
\begin{align}
&\overline{U}_{l}(\overline{\textbf{b}}^{k+1}_{l};\overline{B})\geqslant \phi(\overline{\textbf{W}}^{k+1}_{ l+1},\overline{\textbf{b}}^{k+1}_{l},\overline{\textbf{z}}^{k+1}_{l},\overline{\textbf{a}}^{k+1}_{l},u^k)  \label{eq:lipschitz b back}\\&
    U_{l}(\textbf{b}^{k+1}_{l};B)\geqslant \phi(\textbf{W}^{k+1}_{l},\textbf{b}^{k+1}_{l},\textbf{z}^{k+1}_{l-1},\textbf{a}^{k+1}_{l-1},u^k)  \label{eq:lipschitz b forward}
\end{align} 
\label{lemma:lemma 2}

\end{lemma}
\begin{proof}
Because $\phi(\textbf{W},\textbf{b},\textbf{z},\textbf{a},u)$ is Lipschitz differentiable with respect to $\textbf{b}$ with Lipschitz
coefficient $\nu$ (the definition of Lipschitz differentiablity can be found in \cite{beck2009fast}), we directly apply Lemma 2.1 in \cite{beck2009fast} to $\phi$ to obtain Equations \eqref{eq:lipschitz b back} and \eqref{eq:lipschitz b forward}, respectively.
\end{proof}
\begin{lemma}
\label{lemma:lemma 3}
It holds that for $\forall k\in \mathbb{N}$,
\begin{align}
& L_\rho(\overline{\textbf{W}}^{k+1}_{l+1},\overline{\textbf{b}}^{k+1}_{l+1},\overline{\textbf{z}}^{k+1}_{l+1},\overline{\textbf{a}}^{k+1}_{l+1},u^k)-L_\rho(\overline{\textbf{W}}^{k+1}_{l+1},\overline{\textbf{b}}^{k+1}_{l+1},\overline{\textbf{z}}^{k+1}_{l+1},\overline{\textbf{a}}^{k+1}_l,u^k)\geq \Vert\overline{\tau}_l^{k+1}\circ (\overline{a}^{k+1}_l-a_l^k)^{\circ 2}\Vert_{1}/2(l=1,\cdots,L-1)
\label{eq:overline a optimality}\\
&L_\rho(\overline{\textbf{W}}^{k+1}_{l+1},\overline{\textbf{b}}^{k+1}_{l+1},\overline{\textbf{z}}^{k+1}_{l+1},\overline{\textbf{a}}^{k+1}_l,u^k)\geq L_\rho(\overline{\textbf{W}}^{k+1}_{l+1},\overline{\textbf{b}}^{k+1}_{l+1},\overline{\textbf{z}}^{k+1}_l,\overline{\textbf{a}}^{k+1}_l,u^k)(l=1,\cdots,L-1)\label{eq:overline z optimality}\\
&L_\rho({\textbf{W}}^{k},{\textbf{b}}^{k},{\textbf{z}}^{k},{\textbf{a}}^{k},u^k)- L_\rho({\textbf{W}}^{k},{\textbf{b}}^{k},\overline{\textbf{z}}^{k+1}_L,{\textbf{a}}^k,u^k)\geq (\rho/2)\Vert \overline{z}^{k+1}_L-z^k_L\Vert^2_2 \label{eq:overline zl optimality}\\
& L_\rho(\overline{\textbf{W}}^{k+1}_{l+1},\overline{\textbf{b}}^{k+1}_{l+1},\overline{\textbf{z}}^{k+1}_l,\overline{\textbf{a}}^{k+1}_l,u^k)-L_\rho(\overline{\textbf{W}}^{k+1}_{l+1},\overline{\textbf{b}}^{k+1}_l,\overline{\textbf{z}}^{k+1}_l,\overline{\textbf{a}}^{k+1}_l,u^k)\geq (\nu/2)\Vert \overline{b}^{k+1}_l-b_l^k\Vert^2_2(l=1,\cdots,L-1)\label{eq:overline b optimality}\\
 &L_\rho({\textbf{W}}^{k},{\textbf{b}}^{k},\overline{\textbf{z}}^{k+1}_L,{\textbf{a}}^{k},u^k)- L_\rho({\textbf{W}}^{k},\overline{\textbf{b}}^{k+1}_L,\overline{\textbf{z}}^{k+1}_L,{\textbf{a}}^k,u^k)\geq (\rho/2)\Vert \overline{b}^{k+1}_L-b^k_L\Vert^2_2\label{eq:overline bl optimality}\\
& L_\rho(\overline{\textbf{W}}^{k+1}_{l+1},\overline{\textbf{b}}^{k+1}_l,\overline{\textbf{z}}^{k+1}_l,\overline{\textbf{a}}^{k+1}_l,u^k)-L_\rho(\overline{\textbf{W}}^{k+1}_l,\overline{\textbf{b}}^{k+1}_l,\overline{\textbf{z}}^{k+1}_l,\overline{\textbf{a}}^{k+1}_l,u^k)\geq \Vert\overline{\theta}_l^{k+1}\circ (\overline{W}^{k+1}_l-W_l^k)^{\circ 2}\Vert_{1}/2(l=1,\cdots,L)\label{eq:overline W optimality}\\
& L_\rho({\textbf{W}}^{k+1}_{l-1},{\textbf{b}}^{k+1}_{l-1},{\textbf{z}}^{k+1}_{l-1},{\textbf{a}}^{k+1}_{l-1},u^k)-L_\rho({\textbf{W}}^{k+1}_{l},{\textbf{b}}^{k+1}_{l-1},{\textbf{z}}^{k+1}_{l-1},{\textbf{a}}^{k+1}_{l-1},u^k)\geq \Vert{\theta}_l^{k+1}\circ ({W}^{k+1}_l-\overline{W}_l^{k+1})^{\circ 2}\Vert_{1}/2(l=1,\cdots,L)\label{eq:W optimality}\\
& L_\rho({\textbf{W}}^{k+1}_{l},{\textbf{b}}^{k+1}_{l-1},{\textbf{z}}^{k+1}_{l-1},{\textbf{a}}^{k+1}_{l-1},u^k)-L_\rho({\textbf{W}}^{k+1}_{l},{\textbf{b}}^{k+1}_{l},{\textbf{z}}^{k+1}_{l-1},{\textbf{a}}^{k+1}_{l-1},u^k)\geq (\nu/2)\Vert {b}^{k+1}_l-\overline{b}_l^{k+1}\Vert^2_2(l=1,\cdots,L-1)\label{eq:b optimality}\\
& L_\rho({\textbf{W}}^{k+1},{\textbf{b}}^{k+1}_{L-1},{\textbf{z}}^{k+1}_{L-1},{\textbf{a}}^{k+1},u^k)-L_\rho({\textbf{W}}^{k+1},{\textbf{b}}^{k+1},{\textbf{z}}^{k+1}_{L-1},{\textbf{a}}^{k+1},u^k)\geq (\rho/2)\Vert {b}^{k+1}_L-\overline{b}_L^{k+1}\Vert^2_2\label{eq:bl optimality}\\
 &L_\rho({\textbf{W}}^{k+1}_{l},{\textbf{b}}^{k+1}_{l},{\textbf{z}}^{k+1}_{l-1},{\textbf{a}}^{k+1}_{l-1},u^k)\geq L_\rho({\textbf{W}}^{k+1}_{l},{\textbf{b}}^{k+1}_{l},{\textbf{z}}^{k+1}_l,{\textbf{a}}^{k+1}_{l-1},u^k)(l=1,\cdots,L-1)\label{eq:z optimality}\\
 & L_\rho({\textbf{W}}^{k+1}_{l},{\textbf{b}}^{k+1}_{l},{\textbf{z}}^{k+1}_{l},{\textbf{a}}^{k+1}_{l-1},u^k)-L_\rho({\textbf{W}}^{k+1}_{l},{\textbf{b}}^{k+1}_{l},{\textbf{z}}^{k+1}_{l},{\textbf{a}}^{k+1}_l,u^k)\geq \Vert{\tau}_l^{k+1}\circ ({a}^{k+1}_l-\overline{a}_{l}^{k+1})^{\circ 2}\Vert_{1}/2(l=1,\cdots,L-1)
\label{eq:a optimality}
\end{align}\end{lemma}
\begin{proof}
Essentially, all inequalities can be obtained by applying optimality conditions of updating $\overline{a}^{k+1}_l$, $\overline{z}^{k+1}_l$, $\overline{b}^{k+1}_l$, $\overline{W}^{k+1}_l$, ${W}^{k+1}_l$, ${b}^{k+1}_l$, ${z}^{k+1}_l$ and ${a}^{k+1}_l$, respectively. We only prove Inequality \eqref{eq:overline zl optimality}, \eqref{eq:W optimality}, \eqref{eq:b optimality} and \eqref{eq:z optimality} .This is because Inequalities \eqref{eq:overline a optimality}, \eqref{eq:overline W optimality} and \eqref{eq:a optimality} follow the routine of Inequality \eqref{eq:W optimality}, Inequalities \eqref{eq:overline b optimality}, \eqref{eq:overline bl optimality} and \eqref{eq:bl optimality} follow the routine of Inequality \eqref{eq:b optimality}, and  Inequality \eqref{eq:overline z optimality} follows the routine of Inequality \eqref{eq:z optimality}.\\
\indent Firstly, we focus on Inequality \eqref{eq:overline zl optimality}.
\begin{align}
    \nonumber &L_\rho({\textbf{W}}^{k},{\textbf{b}}^{k},{\textbf{z}}^{k},{\textbf{a}}^{k},u^k)- L_\rho({\textbf{W}}^{k},{\textbf{b}}^{k},\overline{\textbf{z}}^{k+1}_L,{\textbf{a}}^k,u^k)\\\nonumber&=R(z^k_L;y)+(u^k)^T(z^k_L-W^k_La^k_{L-1}-b^k_L)+(\rho/2)\Vert z^k_L-W^k_La^k_{L-1}-b^k_L\Vert^2_2\\\nonumber &-R(\overline{z}^{k+1}_L;y)-(u^k)^T(\overline{z}^{k+1}_L-W^k_La^k_{L-1}-b^k_L)-(\rho/2)\Vert \overline{z}^{k+1}_L-W^k_La^k_{L-1}-b^k_L\Vert^2_2\\&\nonumber=R(z^k_L;y)-R(\overline{z}^{k+1}_L;y)+(u^k)^T(z^k_L-\overline{z}^{k+1}_L)+(\rho/2)\Vert \overline{z}^{k+1}_L-z^k_L\Vert^2_2\\&+ \rho(\overline{z}^{k+1}_L-W^k_La^k_{L-1}-b^k_L)^T(z^k_L-\overline{z}^{k+1}_L)\label{eq:cosine rule}
\end{align}
where the second equality follows from the cosine rule $\Vert z^k_L-W^k_La^k_{L-1}-b^k_L\Vert^2_2-\Vert \overline{z}^{k+1}_L-W^k_La^k_{L-1}-b^k_L\Vert^2_2=\Vert \overline{z}^{k+1}_L-z^k_L\Vert^2_2+(\overline{z}^{k+1}_L-W^k_La^k_{L-1}-b^k_L)^T(z^k_L-\overline{z}^{k+1}_L)$.\\
\indent According to the optimality condition of Equation \eqref{eq:update overline zl},
we have 
$\nabla_{\overline{z}^{k+1}_L}R(\overline{z}^{k+1}_L;y)+u^k+\rho(\overline{z}^{k+1}_L-W^k_La^k_{L-1}-b^k_L)=0$. Because $R(z_L;y)$ is convex and differentiable with regard to $z_L$, its subgradient is also its gradient. According to the definition of subgradient, we have
\begin{align}
    \nonumber R(z^k_L;y)&\geq R(\overline{z}^{k+1}_L;y)+\nabla_{\overline{z}^{k+1}_L}R^T(\overline{z}^{k+1}_L;y)(z^k_L-\overline{z}^{k+1}_L)\\&=R(\overline{z}^{k+1}_L;y)-(u^k+\rho(\overline{z}^{k+1}_L-W^k_La^k_{L-1}-b^k_L))^T(z^k_L-\overline{z}^{k+1}_L)
    \label{eq:zl subgradient}
\end{align}
We introduce Equation \eqref{eq:zl subgradient} into Equation \eqref{eq:cosine rule} to obtain Equation \eqref{eq:overline zl optimality}.\\
\indent Secondly, we focus on Inequality \eqref{eq:W optimality}. The stopping criterion of Algorithm \ref{algo:theta update} shows that
\begin{align}
    \phi(\textbf{W}^{k+1}_l,\textbf{b}^{k+1}_{l-1},\textbf{z}^{k+1}_{l-1},\textbf{a}^{k+1}_{l-1},u^k)\leq P_l(W^{k+1}_{l};\theta^{k+1}_{l}).\label{ineq:ineq1}
\end{align}
Because $\Omega_{W_l}(W_l)$ is convex, according to the definition of subgradient, we have
\begin{align}
  \Omega_l(\overline{W}^{k+1}_{l})\geq  \Omega_l({W}^{k+1}_l)+s^T(\overline{W}_l^{k+1}-W_l^{k+1}). \label{ineq:ineq2}
\end{align}
where $s$ is defined in the premise of Lemma \ref{lemma:lemma 1}. Therefore, we have
\begin{align*}
    & L_\rho({\textbf{W}}^{k+1}_{l-1},{\textbf{b}}^{k+1}_{l-1},{\textbf{z}}^{k+1}_{l-1},{\textbf{a}}^{k+1}_{l-1},u^k)-L_\rho({\textbf{W}}^{k+1}_{l},{\textbf{b}}^{k+1}_{l-1},{\textbf{z}}^{k+1}_{l-1},{\textbf{a}}^{k+1}_{l-1},u^k)\\&=\phi({\textbf{W}}^{k+1}_{l-1},{\textbf{b}}^{k+1}_{l-1},{\textbf{z}}^{k+1}_{l-1},{\textbf{a}}^{k+1}_{l-1},u^k)+\Omega_l(\overline{W}^{k+1}_l)-\phi({\textbf{W}}^{k+1}_{l},{\textbf{b}}^{k+1}_{l-1},{\textbf{z}}^{k+1}_{l-1},{\textbf{a}}^{k+1}_{l-1},u^k)-\Omega_l(W^{k+1}_{l})\ \text{(Definition of $L_\rho$)}\\&\geq \Omega_l(\overline{W}^{k+1}_l)-\Omega_l(W^{k+1}_{l})-\nabla \phi^T_{\overline{W}^{k+1}_l}(W^{k+1}_l-\overline{W}^{k+1}_l)-\Vert\theta_l^{k+1}\circ (W^{k+1}_l-\overline{W}_l^{k+1})^{\circ 2}\Vert_{1}/2(\text{Equation \eqref{ineq:ineq1}})\\&\geq s^T(\overline{W}^{k+1}_l-{W}^{k+1}_l)-\nabla \phi^T_{\overline{W}^{k+1}_l}(W^{k+1}_l-\overline{W}^{k+1}_l)-\Vert\theta_l^{k+1}\circ (W^{k+1}_l-\overline{W}_l^{k+1})^{\circ 2}\Vert_{1}/2(\text{Equation \eqref{ineq:ineq2}})\\ &=(s^T+\nabla \phi^T_{\overline{W}^{k+1}_l})(\overline{W}^{k+1}_l-W^{k+1}_l)-\Vert\theta_l^{k+1}\circ (W^{k+1}_l-\overline{W}_l^{k+1})^{\circ 2}\Vert_{1}/2\\&=\Vert\theta_l^{k+1}\circ (W^{k+1}_l-\overline{W}_l^{k+1})^{\circ 2}\Vert_{1}/2\ (\text{Lemma \ref{lemma:lemma 1}}).
\end{align*}
Thirdly, we focus on Inequality \eqref{eq:b optimality}.
\begin{align*}
    &L_\rho({\textbf{W}}^{k+1}_{l},{\textbf{b}}^{k+1}_{l-1},{\textbf{z}}^{k+1}_{l-1},{\textbf{a}}^{k+1}_{l-1},u^k)-L_\rho({\textbf{W}}^{k+1}_{l},{\textbf{b}}^{k+1}_{l},{\textbf{z}}^{k+1}_{l-1},{\textbf{a}}^{k+1}_{l-1},u^k)\\&=\phi({\textbf{W}}^{k+1}_{l},{\textbf{b}}^{k+1}_{l-1},{\textbf{z}}^{k+1}_{l-1},{\textbf{a}}^{k+1}_{l-1},u^k)-\phi({\textbf{W}}^{k+1}_{l},{\textbf{b}}^{k+1}_{l},{\textbf{z}}^{k+1}_{l-1},{\textbf{a}}^{k+1}_{l-1},u^k)\\&\geq (\nu/2)\Vert {b}^{k+1}_l-\overline{b}_l^{k+1}\Vert^2_2. (\text{Lemma \ref{lemma:lemma 2}})
\end{align*}
Finally, we focus on Equation \eqref{eq:z optimality}. This follows directly  from the optimality of $z^{k+1}_l$ in Equation \eqref{eq:update z}.
\end{proof}
\begin{lemma}
If $\rho>2H$  so that $C_1=\rho/2-H/2-H^2/\rho>0$, then it holds that
\begin{align}
\nonumber &L_\rho(\textbf{W}^{k+1},\textbf{b}^{k+1},\textbf{z}^{k+1}_{L-1},\textbf{a}^{k+1},u^k)-L_\rho(\textbf{W}^{k+1},\textbf{b}^{k+1},\textbf{z}^{k+1},\textbf{a}^{k+1},u^{k+1})\\&\geq  C_1\Vert z_L^{k+1}-\overline{z}_L^{k+1}\Vert^2_2-(H^2/\rho)\Vert\overline{z}_L^{k+1}-{z}_L^{k}\Vert^2_2. \label{eq: lemma4}  
\end{align}
\label{lemma:lemma 4}
\end{lemma}
\begin{proof}
\begin{align*}
    &L_\rho(\textbf{W}^{k+1},\textbf{b}^{k+1},\textbf{z}^{k+1}_{L-1},\textbf{a}^{k+1},u^k)-L_\rho(\textbf{W}^{k+1},\textbf{b}^{k+1},\textbf{z}^{k+1},\textbf{a}^{k+1},u^{k+1})\\&=R(\overline{z}_L^{k+1};y)-R({z}_L^{k+1};y)+(u^{k+1})^T(\overline{z}_L^{k+1}-z_L^{k+1})+(\rho/2)\Vert z_L^{k+1}-\overline{z}_L^{k+1}\Vert^2_2-(1/\rho)\Vert u^{k+1}-u^k\Vert^2_2\\&=R(\overline{z}_L^{k+1};y)-R(z_L^{k+1};y)+\nabla_{z_L^{k+1}} R(z_L^{k+1};y)^T(z_L^{k+1}-\overline{z}_L^{k+1})+(\rho/2)\Vert z_L^{k+1}-\overline{z}_L^{k+1}\Vert^2_2-(1/\rho)\Vert u^{k+1}-u^k\Vert^2_2\text{( Lemma \ref{lemma:z_l optimality})}\\&\geq (-H/2)\Vert z_L^{k+1}-\overline{z}_L^{k+1}\Vert^2_2+(\rho/2)\Vert z_L^{k+1}-\overline{z}_L^{k+1}\Vert^2_2-(1/\rho)\Vert \nabla_{z_L^{k+1}} R(z_L^{k+1};y)-\nabla_{z_L^{k}} R(z_L^{k};y)\Vert^2_2\\&\text{($-R(z_L;y)$ is Lipschitz differentiable, Lemmas \ref{lemma:z_l optimality} and \ref{lemma: R(z_l) lipschitz})}\\&\geq (-H/2)\Vert z_L^{k+1}-\overline{z}_L^{k+1}\Vert^2_2+(\rho/2)\Vert z_L^{k+1}-\overline{z}_L^{k+1}\Vert^2_2-(H^2/\rho)\Vert z_L^{k+1}-{z}_L^{k}\Vert^2_2\text{(Assumption \ref{ass:assumption 2})}\\&\geq(-H/2)\Vert z_L^{k+1}-\overline{z}_L^{k+1}\Vert^2_2+(\rho/2)\Vert z_L^{k+1}-\overline{z}_L^{k+1}\Vert^2_2-(H^2/\rho)\Vert z_L^{k+1}-\overline{z}_L^{k+1}\Vert^2_2-(H^2/\rho)\Vert \overline{z}_L^{k+1}-{z}_L^{k}\Vert^2_2\text{(triangle inequality)}\\&=C_1\Vert z_L^{k+1}-\overline{z}_L^{k+1}\Vert^2_2-(H^2/\rho)\Vert\overline{z}_L^{k+1}-{z}_L^{k}\Vert^2_2.
\end{align*}
We choose $\rho>2H$ to make $C_1>0$.
\end{proof}
\section{Proof of Theorem \ref{thero: theorem 1}}
\label{sec: proof of theorem 1}
Proving Theorem \ref{thero: theorem 1} is equal to proving jointly Theorem \ref{thero: property 1}, \ref{thero: property 2}, and \ref{thero: property 3}, which are elaborated in the following.
\begin{theorem}
Given that Assumptions \ref{ass:assumption 1} and \ref{ass:assumption 2} hold, the dlADMM satisfies Property \ref{pro:property 1}.
\label{thero: property 1}
\end{theorem}
\begin{proof}
There exists $z_L^{'}$ such that $z_L^{'}-W^k_La^k_{L-1}-b^k_L=0$. By Assumption \ref{ass:assumption 2}, we have
\begin{align*}
    &F(\textbf{W}^k,\textbf{b}^k,\{z_l^k\}_{l=1}^{L-1},z^{'}_L,\textbf{a}^k)\geq \min S  > -\infty
\end{align*}
where $S=\{F(\textbf{W},\textbf{b},\textbf{z},\textbf{a}): z_L-W_La_{L-1}-b_L=0\}$. Then we have
\begin{align*}
    &L_\rho(\textbf{W}^k,\textbf{b}^k,\textbf{z}^k,\textbf{a}^k,u^k)\\&=F(\textbf{W}^k,\textbf{b}^k,\textbf{z}^k,\textbf{a}^k)+(u^k)^T(z^k_L-W^k_La^k_{L-1}-b^k_{L})+(\rho/2)\Vert z^k_L-W^k_La^k_{L-1}-b^k_{L}\Vert^2_2\\&=F(\textbf{W}^k,\textbf{b}^k,\textbf{z}^k,\textbf{a}^k)+(u^k)^T(z^k_L-z^{'}_L)+(\rho/2)\Vert z^k_L-W^k_La^k_{L-1}-b^k_{L}\Vert^2_2 \  \text{$(z_L^{'}-W^k_La^k_{L-1}-b^k_L=0)$}\\
    &=F(\textbf{W}^k,\textbf{b}^k,\textbf{z}^k,\textbf{a}^k)+\nabla_{z^k_L} R^T(z^k_L;y)(z^{'}_L-z^k_L)+(\rho/2)\Vert z^k_L-W^k_La^k_{L-1}-b^k_{L}\Vert^2_2\text{(Lemma \ref{lemma:z_l optimality})}\\&=\sum\nolimits_{l=1}^L \Omega_l(W^{k}_l)+(\nu/2)\sum\nolimits_{l=1}^{L-1}(\Vert z^k_l-W^k_la^k_{l-1}-b^k_l\Vert^2_2+\Vert a^k_l-f(z^k_l)\Vert^2_2)+R(z^k_L;y)+\nabla_{z^k_L} R^T(z^k_L;y)(z^{'}_L-z^k_L)\\&+(\rho/2)\Vert z^k_L-W^k_La^k_{L-1}-b^k_{L}\Vert^2_2\text{(The definition of $F$)}\\&\geq\sum\nolimits_{l=1}^L \Omega_l(W^{k}_l)+(\nu/2)\sum\nolimits_{l=1}^{L-1}(\Vert z^k_l-W^k_la^k_{l-1}-b^k_l\Vert^2_2+\Vert a^k_l-f(z^k_l)\Vert^2_2)+R(z^{'}_L;y)+(\rho-H/2)\Vert z^k_L-W^k_La^k_{L-1}-b^k_{L}\Vert^2_2\\&\text{(Lemmas \ref{lemma:z_l optimality} and \ref{lemma: R(z_l) lipschitz}, $R(z_L;y)$ is Lipschitz differentiable)}\\&> -\infty 
\end{align*}
It concludes from Lemma \ref{lemma:lemma 3} and Lemma \ref{lemma:lemma 4} that $L_\rho(\textbf{W}^k,\textbf{b}^k,\textbf{z}^k,\textbf{a}^k,u^k)$ is upper bounded by $L_\rho(\textbf{W}^0,\textbf{b}^0,\textbf{z}^0,\textbf{a}^0,\textbf{u}^0)$ and so are $\sum\nolimits_{l=1}^L \Omega_l(W^{k}_l)+(\nu/2)\sum\nolimits_{l=1}^{L-1}(\Vert z^k_l-W^k_la^k_{l-1}-b^k_l\Vert^2_2+\Vert a^k_l-f(z^k_l)\Vert^2_2)$ and $\Vert z^k_L-W^k_La^k_{L-1}-b^k_L\Vert^2_2$. By Assumption \ref{ass:assumption 2}, $(\textbf{W}^k,\textbf{b}^k,\textbf{z}^k,\textbf{a}^k)$ is bounded. By Lemma \ref{lemma:z_l optimality}, it is obvious that $u^k$ is bounded as well.
\end{proof}
\begin{theorem}
Given that Assumptions \ref{ass:assumption 1} and \ref{ass:assumption 2} hold, the dlADMM satisfies Property \ref{pro:property 2}.
\label{thero: property 2}
\end{theorem}
\begin{proof}
This follows directly from Lemma \ref{lemma:lemma 3} and Lemma \ref{lemma:lemma 4}.
\end{proof}
\begin{theorem}
Given that Assumptions \ref{ass:assumption 1} and \ref{ass:assumption 2} hold, the dlADMM satisfies Property \ref{pro:property 3}.
\label{thero: property 3}
\end{theorem}
\begin{proof}
We know that 
\begin{align*}
&\partial L_\rho(\textbf{W}^{k+1},\textbf{b}^{k+1},\textbf{z}^{k+1},\textbf{a}^{k+1},u^{k+1})\\&=(\partial_{\textbf{W}^{k+1}} L_\rho,\nabla_{\textbf{b}^{k+1}} L_\rho,\partial_{\textbf{z}^{k+1}} L_\rho,\nabla_{\textbf{a}^{k+1}}L_\rho,\nabla_{u^{k+1}} L_\rho)
\end{align*}
where $\partial_{\textbf{W}^{k+1}} L_\rho=\{\partial_{{W_l}^{k+1}} L_\rho\}_{l=1}^L$, $\nabla_{\textbf{b}^{k+1}} L_\rho=\{\nabla_{{b_l}^{k+1}}L_\rho\}_{l=1}^L$, $\partial_{\textbf{z}^{k+1}} L_\rho=\{\partial_{{z_l}^{k+1}} L_\rho\}_{l=1}^L$, and $\nabla_{\textbf{a}^{k+1}} L_\rho=\{\nabla_{{a_l}^{k+1}} L_\rho\}_{l=1}^{L-1}$ . To prove Property \ref{pro:property 3}, we need to give an upper bound of $\partial_{\textbf{W}^{k+1}} L_\rho,\nabla_{\textbf{b}^{k+1}} L_\rho,\partial_{\textbf{z}^{k+1}} L_\rho,\nabla_{\textbf{a}^{k+1}}L_\rho$ and $\nabla_{u^{k+1}} L_\rho$ by a linear combination of $\Vert \textbf{W}^{k+1}-\overline{\textbf{W}}^{k+1}\Vert$,  $\Vert\textbf{b}^{k+1}-\overline{\textbf{b}}^{k+1}\Vert$, $\Vert\textbf{z}^{k+1}-\overline{\textbf{z}}^{k+1}\Vert$, $\Vert\textbf{a}^{k+1}-\overline{\textbf{a}}^{k+1}\Vert$ and $\Vert \textbf{z}_l^{k+1}-\textbf{z}_l^{k}\Vert$.\\
\indent For $W^{k+1}_l(l<L)$,
\begin{align*}
    \partial_{W^{k+1}_l} L_\rho&=\partial\Omega_l(W^{k+1}_l)+\nabla _{W_l^{k+1}}\phi(\textbf{W}^{k+1},\textbf{b}^{k+1},\textbf{z}^{k+1},\textbf{a}^{k+1},u^{k+1})\\&=\partial\Omega_l(W^{k+1}_l)+\nu(W^{k+1}_la^{k+1}_{l-1}+b^{k+1}_l-z^{k+1}_l)(a^{k+1}_{l-1})^T\\&=\partial\Omega_l(W^{k+1}_l)+\nabla _{W_l^{k+1}}\phi(\textbf{W}^{k+1}_l,\textbf{b}^{k+1}_{l},\textbf{z}^{k+1}_{l},\textbf{a}^{k+1}_{l-1},u^{k})\\&=\partial\Omega_l(W^{k+1}_l)+\nabla _{\overline{W}_l^{k+1}}\phi(\textbf{W}^{k+1}_{l-1},\textbf{b}^{k+1}_{l-1},\textbf{z}^{k+1}_{l-1},\textbf{a}^{k+1}_{l-1},u^{k})+\theta^{k+1}_l\circ(W^{k+1}_l-\overline{W}^{k+1}_l)- \theta^{k+1}_l\circ(W^{k+1}_l-\overline{W}^{k+1}_l)\\&-\nabla _{\overline{W}_l^{k+1}}\phi(\textbf{W}^{k+1}_{l-1},\textbf{b}^{k+1}_{l-1},\textbf{z}^{k+1}_{l-1},\textbf{a}^{k+1}_{l-1},u^{k})+\nabla_{W_l^{k+1}}\phi(\textbf{W}^{k+1}_l,\textbf{b}^{k+1}_{l},\textbf{z}^{k+1}_{l},\textbf{a}^{k+1}_{l-1},u^{k})
    \\&=\partial\Omega_l(W^{k+1}_l)+\nabla_{\overline{W}_l^{k+1}}\phi(\textbf{W}^{k+1}_{l-1},\textbf{b}^{k+1}_{l-1},\textbf{z}^{k+1}_{l-1},\textbf{a}^{k+1}_{l-1},u^{k})+\theta^{k+1}_l\circ(W^{k+1}_l-\overline{W}^{k+1}_l)- \theta^{k+1}_l\circ(W^{k+1}_l-\overline{W}^{k+1}_l)\\&+\nu(W^{k+1}_la^{k+1}_{l-1}+b^{k+1}_l-z^{k+1}_l)(a^{k+1}_{l-1})^T-\nu(\overline{W}^{k+1}_{l}a^{k+1}_{l-1}+\overline{b}^{k+1}_l-\overline{z}^{k+1}_l)(a^{k+1}_{l-1})^T
\end{align*}
Because
\begin{align*}
&\Vert - \theta^{k+1}_l\circ(W^{k+1}_l-\overline{W}^{k+1}_l)+\nu(W^{k+1}_la^{k+1}_{l-1}+b^{k+1}_l-z^{k+1}_l)(a^{k+1}_{l-1})^T-\nu(\overline{W}^{k+1}_{l}a^{k+1}_{l-1}+\overline{b}^{k+1}_l-\overline{z}^{k+1}_l)(a^{k+1}_{l-1})^T\Vert\\&=\Vert - \theta^{k+1}_l\circ(W^{k+1}_l-\overline{W}^{k+1}_l)+\nu(W^{k+1}_l-\overline{W}^{k+1}_l)a^{k+1}_{l-1}(a^{k+1}_{l-1})^T+\nu(b^{k+1}_l-\overline{b}^{k+1}_l)(a^{k+1}_{l-1})^T-\nu(z^{k+1}_l-\overline{z}^{k+1}_l)(a^{k+1}_{l-1})^T\Vert\\&\leq\Vert \theta^{k+1}_l\circ(W^{k+1}_l-\overline{W}^{k+1}_l)\Vert+\nu\Vert(W^{k+1}_l-\overline{W}^{k+1}_l)a^{k+1}_{l-1}(a^{k+1}_{l-1})^T\Vert+\nu\Vert(b^{k+1}_l-\overline{b}^{k+1}_l)(a^{k+1}_{l-1})^T\Vert+\nu\Vert(z^{k+1}_l-\overline{z}^{k+1}_l)(a^{k+1}_{l-1})^T\Vert\\&\text{(triangle inequality)}\\&\leq\Vert \theta^{k+1}_l\circ (W^{k+1}_l-\overline{W}^{k+1}_l)\Vert+\nu\Vert W^{k+1}_l-\overline{W}^{k+1}_l\Vert\Vert a^{k+1}_{l-1}\Vert\Vert a^{k+1}_{l-1}\Vert+\nu\Vert b^{k+1}_l-\overline{b}^{k+1}_l\Vert\Vert a^{k+1}_{l-1}\Vert+\nu\Vert z^{k+1}_l-\overline{z}^{k+1}_l\Vert\Vert a^{k+1}_{l-1}\Vert\\& \text{(Cauchy-Schwarz inequality)}
\end{align*}
and the optimality condition of Equation \eqref{eq:update W} yields
\begin{align*}
    0&\in\partial\Omega_l(W^{k+1}_l)+\nabla_{\overline{W}_l^{k+1}}\phi(\textbf{W}^{k+1}_{l-1},\textbf{b}^{k+1}_{l-1},\textbf{z}^{k+1}_{l-1},\textbf{a}^{k+1}_{l-1},u^{k})+\theta^{k+1}_l\circ(W^{k+1}_l-\overline{W}^{k+1}_l)
\end{align*}
Because  $a^{k+1}_{l-1}$ is bounded by Property \ref{pro:property 1}, $\Vert\partial _{W_l^{k+1} }L_\rho\Vert$ can be upper bounded by a linear combination of  $\Vert W^{k+1}_l-\overline{W}^{k+1}_l\Vert$, $\Vert b^{k+1}_l-\overline{b}^{k+1}_l\Vert$ and $\Vert z^{k+1}_l-\overline{z}^{k+1}_l\Vert$.\\
\indent For $W^{k+1}_L$,
\begin{align*}
    \partial_{W^{k+1}_L} L_\rho&=\partial\Omega_L(W^{k+1}_L)+\nabla_{W_L^{k+1}}\phi(\textbf{W}^{k+1},\textbf{b}^{k+1},\textbf{z}^{k+1},\textbf{a}^{k+1},u^{k+1})\\&=\partial\Omega_L(W^{k+1}_L)+\nabla _{\overline{W}_L^{k+1}}\phi(\textbf{W}^{k+1}_{L-1},\textbf{b}^{k+1}_{L-1},\textbf{z}^{k+1}_{L-1},\textbf{a}^{k+1},u^{k})+\theta^{k+1}_L\circ(W^{k+1}_L-\overline{W}^{k+1}_L)- \theta^{k+1}_L\circ(W^{k+1}_L-\overline{W}^{k+1}_L)\\&-\nabla _{\overline{W}_L^{k+1}}\phi(\textbf{W}^{k+1}_{L-1},\textbf{b}^{k+1}_{L-1},\textbf{z}^{k+1}_{L-1},\textbf{a}^{k+1},u^{k})+\nabla_{W_L^{k+1}}\phi(\textbf{W}^{k+1},\textbf{b}^{k+1},\textbf{z}^{k+1},\textbf{a}^{k+1},u^{k+1})
    \\&=\partial\Omega_L(W^{k+1}_L)+\nabla_{W_L^{k+1}}\phi(\textbf{W}^{k+1}_{L-1},\textbf{b}^{k+1}_{L-1},\textbf{z}^{k+1}_{L-1},\textbf{a}^{k+1},u^{k})+\theta^{k+1}_L\circ(W^{k+1}_L-\overline{W}^{k+1}_L)- \theta^{k+1}_L\circ(W^{k+1}_L-\overline{W}^{k+1}_L)\\&+\rho(W^{k+1}_La^{k+1}_{L-1}+b^{k+1}_L-z^{k+1}_L-u^{k+1}/\rho)(a^{k+1}_{L-1})^T-\rho(\overline{W}^{k+1}_{L}a^{k+1}_{L-1}+\overline{b}^{k+1}_L-\overline{z}^{k+1}_L-u^k/\rho)(a^{k+1}_{L-1})^T
\end{align*}
Because
\begin{align*}
&\Vert - \theta^{k+1}_L\circ(W^{k+1}_L-\overline{W}^{k+1}_L)+\rho(W^{k+1}_La^{k+1}_{L-1}+b^{k+1}_L-z^{k+1}_L-u^{k+1}/\rho)(a^{k+1}_{L-1})^T-\rho(\overline{W}^{k+1}_{L}a^{k+1}_{L-1}+\overline{b}^{k+1}_L-\overline{z}^{k+1}_L-u^k/\rho)(a^{k+1}_{L-1})^T\Vert\\&=\Vert - \theta^{k+1}_L\circ(W^{k+1}_L-\overline{W}^{k+1}_L)+\rho(W^{k+1}_L-\overline{W}^{k+1}_L)a^{k+1}_{L-1}(a^{k+1}_{L-1})^T+\rho(b^{k+1}_L-\overline{b}^{k+1}_L)(a^{k+1}_{L-1})^T-\rho(z^{k+1}_L-\overline{z}^{k+1}_L)(a^{k+1}_{L-1})^T-(u^{k+1}-{u}^{k})(a^{k+1}_{L-1})^T\Vert\\&\leq\Vert \theta^{k+1}_L\circ(W^{k+1}_L-\overline{W}^{k+1}_L)\Vert+\rho\Vert(W^{k+1}_L-\overline{W}^{k+1}_L)a^{k+1}_{L-1}(a^{k+1}_{L-1})^T\Vert+\rho\Vert(b^{k+1}_L-\overline{b}^{k+1}_L)(a^{k+1}_{L-1})^T\Vert+\rho\Vert(z^{k+1}_L-\overline{z}^{k+1}_L)(a^{k+1}_{L-1})^T\Vert\\&+\Vert(u^{k+1}-{u}^{k})(a^{k+1}_{L-1})^T\Vert\text{(triangle inequality)}\\&\leq\Vert \theta^{k+1}_L\circ (W^{k+1}_L-\overline{W}^{k+1}_L)\Vert+\rho\Vert W^{k+1}_L-\overline{W}^{k+1}_L\Vert\Vert a^{k+1}_{L-1}\Vert\Vert a^{k+1}_{L-1}\Vert+\rho\Vert b^{k+1}_L-\overline{b}^{k+1}_L\Vert\Vert a^{k+1}_{L-1}\Vert+\rho\Vert z^{k+1}_L-\overline{z}^{k+1}_L\Vert\Vert a^{k+1}_{L-1}\Vert+H\Vert z_L^{k+1}-z_L^k\Vert \Vert a^{k+1}_{L-1}\Vert\\& \text{(Cauchy-Schwarz inequality, Lemma \ref{lemma:z_l optimality}, $R(z_L;y)$ is Lipschitz differentiable)}
\end{align*}
and the optimality condition of Equation \eqref{eq:update W} yields
\begin{align*}
    0&\in\partial\Omega_L(W^{k+1}_L)+\nabla \phi_{\overline{W}_L^{k+1}}(\textbf{W}^{k+1}_{L-1},\textbf{b}^{k+1}_{L-1},\textbf{z}^{k+1}_{L-1},\textbf{a}^{k+1},u^{k})+\theta^{k+1}_L\circ(W^{k+1}_L-\overline{W}^{k+1}_L)
\end{align*}
Because  $a^{k+1}_{L-1}$ is bounded by Property \ref{pro:property 1}, $\Vert\partial _{W_L^{k+1} }L_\rho\Vert$ can be upper bounded by a linear combination of  $\Vert W^{k+1}_L-\overline{W}^{k+1}_L\Vert$, $\Vert b^{k+1}_L-\overline{b}^{k+1}_L\Vert$, $\Vert z^{k+1}_L-\overline{z}^{k+1}_L\Vert$ and $\Vert z^{k+1}_L-{z}^{k}_L\Vert$.\\
\indent For $b^{k+1}_l(l<L)$,
\begin{align*}
    \nabla _{b^{k+1}_l}L_\rho&=\nabla _{b^{k+1}_l}\phi(\textbf{W}^{k+1},\textbf{b}^{k+1},\textbf{z}^{k+1},\textbf{a}^{k+1},u^{k+1})\\&=\nu(W^{k+1}_{l}a^{k+1}_{l-1}+b^{k+1}_l-z^{k+1}_l)\\&=\nabla_{b^{k+1}_l}\phi(\textbf{W}^{k+1}_l,\textbf{b}^{k+1}_{l},\textbf{z}^{k+1}_{l},\textbf{a}^{k+1}_{l},u^k)\\&=\nabla_{\overline{b}^{k+1}_l}\phi({\textbf{W}}^{k+1}_{l},{\textbf{b}}^{k+1}_{l-1},{\textbf{z}}^{k+1}_{l-1},{\textbf{a}}^{k+1}_{l-1},u^k)+\nu(b^{k+1}_l-\overline{b}^{k+1}_l)-\nu(b^{k+1}_l-\overline{b}^{k+1}_l)-\nabla_{\overline{b}^{k+1}_l}\phi({\textbf{W}}^{k+1}_{l},{\textbf{b}}^{k+1}_{l-1},{\textbf{z}}^{k+1}_{l-1},{\textbf{a}}^{k+1}_{l-1},u^k)\\&+\nabla_{b^{k+1}_l}\phi(\textbf{W}^{k+1}_l,\textbf{b}^{k+1}_{l},\textbf{z}^{k+1}_{l},\textbf{a}^{k+1}_{l},u^k)\\&=\nabla_{\overline{b}^{k+1}_l}\phi({\textbf{W}}^{k+1}_{l},{\textbf{b}}^{k+1}_{l-1},{\textbf{z}}^{k+1}_{l-1},{\textbf{a}}^{k+1}_{l-1},u^k)+\nu(b^{k+1}_l-\overline{b}^{k+1}_l)-\nu(b^{k+1}_l-\overline{b}^{k+1}_l)-\nu({W}^{k+1}_{l}a^{k+1}_{l-1}+\overline{b}^{k+1}_l-\overline{z}^{k+1}_l)\\&+\nu(W^{k+1}_la^{k+1}_{l-1}+b^{k+1}_l-z^{k+1}_l)\\&
    =\nabla_{\overline{b}^{k+1}_l}\phi({\textbf{W}}^{k+1}_{l},{\textbf{b}}^{k+1}_{l-1},{\textbf{z}}^{k+1}_{l-1},{\textbf{a}}^{k+1}_{l-1},u^k)+\nu(b^{k+1}_l-\overline{b}^{k+1}_l)+\nu(\overline{z}^{k+1}_l-z^{k+1}_l)
    \end{align*}
    The optimality condition of Equation \eqref{eq:update b} yields
    \begin{align*}
        0\in \nabla_{\overline{b}^{k+1}_l}\phi({\textbf{W}}^{k+1}_{l},{\textbf{b}}^{k+1}_{l-1},{\textbf{z}}^{k+1}_{l-1},{\textbf{a}}^{k+1}_{l-1},u^k)+\nu(b^{k+1}_l-\overline{b}^{k+1}_l)
    \end{align*}
    Therefore, $\Vert\nabla_{b^{k+1}_l} L_\rho\Vert$ is linearly independent on $\Vert z^{k+1}_l-\overline{z}^{k+1}_l\Vert$. \\
    \indent For $b^{k+1}_{L}$,
    \begin{align*}
    \nabla _{b^{k+1}_L}L_\rho&=\nabla_{b^{k+1}_L}\phi(\textbf{W}^{k+1},\textbf{b}^{k+1},\textbf{z}^{k+1},\textbf{a}^{k+1},u^{k+1})\\&=\nabla_{\overline{b}^{k+1}_L}\phi({\textbf{W}}^{k+1},{\textbf{b}}^{k+1}_{L-1},{\textbf{z}}^{k+1}_{L-1},{\textbf{a}}^{k+1},u^k)+\rho(b^{k+1}_L-\overline{b}^{k+1}_L)-\rho(b^{k+1}_L-\overline{b}^{k+1}_L)-\nabla_{\overline{b}^{k+1}_L}\phi({\textbf{W}}^{k+1},{\textbf{b}}^{k+1}_{L-1},{\textbf{z}}^{k+1}_{L-1},{\textbf{a}}^{k+1},u^k)\\&+\nabla_{b^{k+1}_L}\phi(\textbf{W}^{k+1},\textbf{b}^{k+1},\textbf{z}^{k+1},\textbf{a}^{k+1},u^{k+1})\\&=\nabla_{\overline{b}^{k+1}_L}\phi({\textbf{W}}^{k+1},{\textbf{b}}^{k+1}_{L-1},{\textbf{z}}^{k+1}_{L-1},{\textbf{a}}^{k+1},u^k)+\rho(b^{k+1}_L-\overline{b}^{k+1}_L)-\rho(b^{k+1}_L-\overline{b}^{k+1}_L)-\rho(W^{k+1}_La^{k+1}_L+\overline{b}^{k+1}_L-\overline{z}^{k+1}_L-u^k/\rho)\\&+\rho(W^{k+1}_La^{k+1}_L+{b}^{k+1}_L-{z}^{k+1}_L-u^{k+1}/\rho)\\&=\nabla_{\overline{b}^{k+1}_L}\phi({\textbf{W}}^{k+1},{\textbf{b}}^{k+1}_{L-1},{\textbf{z}}^{k+1}_{L-1},{\textbf{a}}^{k+1},u^k)+\rho(b^{k+1}_L-\overline{b}^{k+1}_L)+\rho(\overline{z}^{k+1}_L-z^{k+1}_L)+u^k-u^{k+1}
    \end{align*}
    Because
    \begin{align*}
        &\Vert\rho(\overline{z}^{k+1}_L-z^{k+1}_L)+u^k-u^{k+1}\Vert\\&\leq \rho\Vert\overline{z}^{k+1}_L-z^{k+1}_L\Vert+\Vert u^k-u^{k+1}\Vert\text{(triangle inequality)}\\&=\rho\Vert\overline{z}^{k+1}_L-z^{k+1}_L\Vert+\Vert \nabla_{z_L^{k}}R(z_L^{k};y)-\nabla_{z_L^{k+1}}R(z_L^{k+1};y)\Vert\text{(Lemma \ref{lemma:z_l optimality})}\\&\leq \rho \Vert\overline{z}^{k+1}_L-z^{k+1}_L\Vert+H\Vert z_L^{k}-z_L^{k+1}\Vert\text{($R(z_L;y)$ is Lipschitz differentiable)}
    \end{align*}
    and the optimality condition of Equation \eqref{eq:update bl} yields
    \begin{align*}
        0\in \nabla_{\overline{b}^{k+1}_L}\phi({\textbf{W}}^{k+1},{\textbf{b}}^{k+1}_{L-1},{\textbf{z}}^{k+1}_{L-1},{\textbf{a}}^{k+1},u^k)+\rho(b^{k+1}_L-\overline{b}^{k+1}_L)
    \end{align*}
    Therefore, $\Vert\nabla_{b^{k+1}_L} L_\rho\Vert$ is upper bounded by a combination of $\Vert z^{k+1}_L-\overline{z}^{k+1}_L\Vert$ and $\Vert z^{k+1}_L-{z}^{k}_L\Vert$. \\
    \indent For $z^{k+1}_l(l< L)$,
    \begin{align*}
        \partial_{z^{k+1}_l} L_\rho&=\partial_{z^{k+1}_l}\phi(\textbf{W}^{k+1},\textbf{b}^{k+1},\textbf{z}^{k+1},\textbf{a}^{k+1},u^{k+1})\\&=\nu(z^{k+1}_l-W^{k+1}_la^{k+1}_{l-1}-b^{k+1}_l)+\nu \partial f_l(z^{k+1}_l)\circ(f(z^{k+1}_l)-a^{k+1}_l)\\&=\partial_{z^{k+1}_l}\phi(\textbf{W}_l^{k+1},\textbf{b}_l^{k+1},\textbf{z}_l^{k+1},\textbf{a}_l^{k+1},u^{k})\\&=\partial_{z^{k+1}_l}\phi(\textbf{W}_l^{k+1},\textbf{b}_l^{k+1},\textbf{z}_l^{k+1},\textbf{a}_l^{k+1},u^{k})-\partial_{z^{k+1}_l}\phi(\textbf{W}_l^{k+1},\textbf{b}_l^{k+1},\textbf{z}_l^{k+1},\textbf{a}_{l-1}^{k+1},u^{k})+\partial_{z^{k+1}_l}\phi(\textbf{W}_l^{k+1},\textbf{b}_l^{k+1},\textbf{z}_l^{k+1},\textbf{a}_{l-1}^{k+1},u^{k})\\&=\nu\partial f_l(z^{k+1}_l)\circ(\overline{a}_l^{k+1}-a^{k+1}_l)+\partial_{z^{k+1}_l}\phi(\textbf{W}_l^{k+1},\textbf{b}_l^{k+1},\textbf{z}_l^{k+1},\textbf{a}_{l-1}^{k+1},u^{k})
    \end{align*}
    because $z^{k+1}_l$ is bounded and $f_l(z_l)$ is continuous and hence $f_l(z^{k+1}_l)$ is bounded, and the optimality condition of Equation \eqref{eq:update z} yields
    \begin{align*}
        0\in \partial_{z^{k+1}_l}\phi(\textbf{W}_l^{k+1},\textbf{b}_l^{k+1},\textbf{z}_l^{k+1},\textbf{a}_{l-1}^{k+1},u^{k})
    \end{align*}
    \indent Therefore, $\Vert \partial_{z^{k+1}_l}L_\rho\Vert$ is upper bounded by $\Vert a^{k+1}_l-\overline{a}^{k+1}_l\Vert$.\\
    \indent For $z^{k+1}_L$,
    \begin{align*}
        \nabla_{z^{k+1}_L}L_\rho&=\nabla_{z^{k+1}_L} R(z^{k+1}_L;y)+\nabla_{z^{k+1}_L}\phi(\textbf{W}^{k+1},\textbf{b}^{k+1},\textbf{z}^{k+1},\textbf{a}^{k+1},u^{k+1})\\&=\nabla_{z^{k+1}_L} R(z^{k+1}_L;y)+\nabla_{z^{k+1}_L}\phi(\textbf{W}^{k+1},\textbf{b}^{k+1},\textbf{z}^{k+1},\textbf{a}^{k+1},u^{k})-\nabla_{z^{k+1}_L}\phi(\textbf{W}^{k+1},\textbf{b}^{k+1},\textbf{z}^{k+1},\textbf{a}^{k+1},u^{k})\\&+\nabla_{z^{k+1}_L}\phi(\textbf{W}^{k+1},\textbf{b}^{k+1},\textbf{z}^{k+1},\textbf{a}^{k+1},u^{k+1})\\&=\nabla_{z^{k+1}_L} R(z^{k+1}_L;y)+\nabla_{z^{k+1}_L}\phi(\textbf{W}^{k+1},\textbf{b}^{k+1},\textbf{z}^{k+1},\textbf{a}^{k+1},u^{k})+u^{k+1}-u^k
    \end{align*}
    The optimality condition of Equation \eqref{eq:update zl} yields
    \begin{align*}
        0\in\nabla_{z^{k+1}_L} R(z^{k+1}_L;y)+\nabla_{z^{k+1}_L}\phi(\textbf{W}^{k+1},\textbf{b}^{k+1},\textbf{z}^{k+1},\textbf{a}^{k+1},u^{k})
    \end{align*}
     and
     \begin{align*}
             &\Vert u^{k+1}-u^k\Vert= \Vert \nabla_{z^{k+1}_L}R(z^{k+1}_L;y)-\nabla_{z^{k}_L}R(z^{k}_L;y)\Vert \ \text{(Lemma \ref{lemma:z_l optimality})}\leq H \Vert z^{k+1}_L-z^k_L\Vert\\&\text{(Cauchy-Schwarz inequality, $R(z_L;y)$ is Lipschitz differentiable)}
     \end{align*}
 Therefore $\Vert\nabla_{z^{k+1}_L}L_\rho\Vert$ is upper bounded by $\Vert z^{k+1}_L-z^{k}_L\Vert$.\\
\indent For $a^{k+1}_l(l<L-1)$,
\begin{align*}
    &\nabla_{a^{k+1}_l}\phi(\textbf{W}^{k+1},\textbf{b}^{k+1},\textbf{z}^{k+1},\textbf{a}^{k+1},u^{k+1})\\&=\nu (W^{k+1}_{l+1})^T(W^{k+1}_{l+1}a^{k+1}_l+b^{k+1}_{l+1}-z^{k+1}_{l+1})+\nu(a^{k+1}_l-f_l(z^{k+1}_l))\\&=\nabla_{a^{k+1}_l}\phi(\textbf{W}_{l+1}^{k+1},\textbf{b}^{k+1}_{l+1},\textbf{z}^{k+1}_{l+1},\textbf{a}^{k+1}_{l},u^{k})\\&=\tau^{k+1}_l\circ(a^{k+1}_l-\overline{a}^{k+1}_l)+\nabla_{\overline{a}^{k+1}_l}\phi(\textbf{W}_{l}^{k+1},\textbf{b}^{k+1}_{l},\textbf{z}^{k+1}_{l},\textbf{a}^{k+1}_{l-1},u^{k})-\tau^{k+1}_l\circ(a^{k+1}_l-\overline{a}^{k+1}_l)-\nabla_{\overline{a}^{k+1}_l}\phi(\textbf{W}_{l}^{k+1},\textbf{b}^{k+1}_{l},\textbf{z}^{k+1}_{l},\textbf{a}^{k+1}_{l-1},u^{k})\\&+\nabla_{a^{k+1}_l}\phi(\textbf{W}_{l+1}^{k+1},\textbf{b}^{k+1}_{l+1},\textbf{z}^{k+1}_{l+1},\textbf{a}^{k+1}_{l},u^{k})\\&=\tau^{k+1}_l\circ(a^{k+1}_l-\overline{a}^{k+1}_l)+\nabla_{\overline{a}^{k+1}_l}\phi(\textbf{W}_{l}^{k+1},\textbf{b}^{k+1}_{l},\textbf{z}^{k+1}_{l},\textbf{a}^{k+1}_{l-1},u^{k})-\tau^{k+1}_l\circ(a^{k+1}_l-\overline{a}^{k+1}_l)-\nu (\overline{W}^{k+1}_{l+1})^T(\overline{W}^{k+1}_{l+1}\overline{a}^{k+1}_l+\overline{b}^{k+1}_{l+1}-\overline{z}^{k+1}_{l+1})\\&-\nu(\overline{a}^{k+1}_l-f_l(z^{k+1}_l))+\nu (W^{k+1}_{l+1})^T(W^{k+1}_{l+1}a^{k+1}_l+b^{k+1}_{l+1}-z^{k+1}_{l+1})+\nu(a^{k+1}_l-f_l(z^{k+1}_l))
\end{align*}
Because
\begin{align*}
    &\Vert-\tau^{k+1}_l\circ(a^{k+1}_l-\overline{a}^{k+1}_l)-\nu (\overline{W}^{k+1}_{l+1})^T(\overline{W}^{k+1}_{l+1}\overline{a}^{k+1}_l+\overline{b}^{k+1}_{l+1}-\overline{z}^{k+1}_{l+1})-\nu(\overline{a}^{k+1}_l-f_l(z^{k+1}_l))\\&+\nu (W^{k+1}_{l+1})^T(W^{k+1}_{l+1}a^{k+1}_l+b^{k+1}_{l+1}-z^{k+1}_{l+1})+\nu(a^{k+1}_l-f_l(z^{k+1}_l))\Vert\\&\leq \Vert\tau^{k+1}_l\circ(a^{k+1}_l-\overline{a}^{k+1}_l)\Vert+\nu\Vert a^{k+1}_l-\overline{a}^{k+1}_l\Vert+\nu\Vert (W^{k+1}_{l+1})^Tb^{k+1}_{l+1}-(\overline{W}^{k+1}_{l+1})^T\overline{b}^{k+1}_{l+1} \Vert\\&+\nu\Vert (W^{k+1}_{l+1})^Tz^{k+1}_{l+1}-(\overline{W}^{k+1}_{l+1})^T\overline{z}^{k+1}_{l+1}\Vert +\nu\Vert (W^{k+1}_{l+1})^TW^{k+1}_{l+1}a^{k+1}_l-(\overline{W}^{k+1}_{l+1})^T\overline{W}^{k+1}_{l+1}\overline{a}^{k+1}_l\Vert\text{(triangle inequality)}
\end{align*}
Then we need to show that $\nu\Vert (W^{k+1}_{l+1})^Tb^{k+1}_{l+1}-(\overline{W}^{k+1}_{l+1})^T\overline{b}^{k+1}_{l+1} \Vert$, $\nu\Vert (W^{k+1}_{l+1})^Tz^{k+1}_{l+1}-(\overline{W}^{k+1}_{l+1})^T\overline{z}^{k+1}_{l+1} \Vert$, and $\nu\Vert (W^{k+1}_{l+1})^TW^{k+1}_{l+1}a^{k+1}_l-(\overline{W}^{k+1}_{l+1})^T\overline{W}^{k+1}_{l+1}\overline{a}^{k+1}_l\Vert$ are upper bounded by $\Vert \textbf{W}^{k+1}-\overline{\textbf{W}}^{k+1}\Vert$,  $\Vert\textbf{b}^{k+1}-\overline{\textbf{b}}^{k+1}\Vert$, $\Vert\textbf{z}^{k+1}-\overline{\textbf{z}}^{k+1}\Vert$, and $\Vert\textbf{a}^{k+1}-\overline{\textbf{a}}^{k+1}\Vert$.
\begin{align*}
    &\nu\Vert (W^{k+1}_{l+1})^Tb^{k+1}_{l+1}-(\overline{W}^{k+1}_{l+1})^T\overline{b}^{k+1}_{l+1} \Vert\\&=\nu\Vert (W^{k+1}_{l+1})^Tb^{k+1}_{l+1}-(\overline{W}^{k+1}_{l+1})^Tb^{k+1}_{l+1}+(\overline{W}^{k+1}_{l+1})^Tb^{k+1}_{l+1}-(\overline{W}^{k+1}_{l+1})^T\overline{b}^{k+1}_{l+1} \Vert\\&\leq \nu\Vert b^{k+1}_{l+1}\Vert\Vert W^{k+1}_{l+1}-\overline{W}^{k+1}_{l+1}\Vert+\nu\Vert \overline{W}^{k+1}_{l+1}\Vert\Vert b^{k+1}_{l+1}-\overline{b}^{k+1}_{l+1}\Vert\\&\text{(triangle inequality, Cauchy-Schwarz inequality)}
\end{align*}
Because $\Vert b^{k+1}_{l+1}\Vert$ and $\Vert\overline{W}^{k+1}_{l+1}\Vert$ are upper bounded, $\nu\Vert (W^{k+1}_{l+1})^Tb^{k+1}_{l+1}-(\overline{W}^{k+1}_{l+1})^T\overline{b}^{k+1}_{l+1} \Vert$ is therefore upper bounded by a combination of $\Vert W^{k+1}_{l+1}-\overline{W}^{k+1}_{l+1}\Vert$ and $\Vert b^{k+1}_{l+1}-\overline{b}^{k+1}_{l+1}\Vert$.\\
Similarly, $\nu\Vert (W^{k+1}_{l+1})^Tz^{k+1}_{l+1}-(\overline{W}^{k+1}_{l+1})^T\overline{z}^{k+1}_{l+1} \Vert$ is  upper bounded by a combination of $\Vert W^{k+1}_{l+1}-\overline{W}^{k+1}_{l+1}\Vert$ and $\Vert z^{k+1}_{l+1}-\overline{z}^{k+1}_{l+1}\Vert$.\\
\begin{align*}
    &\nu\Vert (W^{k+1}_{l+1})^TW^{k+1}_{l+1}a^{k+1}_l-(\overline{W}^{k+1}_{l+1})^T\overline{W}^{k+1}_{l+1}\overline{a}^{k+1}_l\Vert\\&=\nu\Vert(W^{k+1}_{l+1})^TW^{k+1}_{l+1}a^{k+1}_l-(W^{k+1}_{l+1})^TW^{k+1}_{l+1}\overline{a}^{k+1}_l+(W^{k+1}_{l+1})^TW^{k+1}_{l+1}\overline{a}^{k+1}_l-(W^{k+1}_{l+1})^T\overline{W}^{k+1}_{l+1}\overline{a}^{k+1}_l+(W^{k+1}_{l+1})^T\overline{W}^{k+1}_{l+1}\overline{a}^{k+1}_l-(\overline{W}^{k+1}_{l+1})^T\overline{W}^{k+1}_{l+1}\overline{a}^{k+1}_l\Vert\\&\leq \nu\Vert(W^{k+1}_{l+1})^TW^{k+1}_{l+1}(a^{k+1}_l-\overline{a}^{k+1}_l)\Vert+\nu\Vert (W^{k+1}_{l+1})^T(W^{k+1}_{l+1}-\overline{W}^{k+1}_{l+1})\overline{a}^{k+1}_l\Vert+\nu\Vert(W^{k+1}_{l+1}-\overline{W}^{k+1}_{l+1})^T\overline{W}^{k+1}_{l+1}\overline{a}^{k+1}_l\Vert\text{(triangle inequality)}\\&\leq \nu\Vert W^{k+1}_{l+1}\Vert\Vert W^{k+1}_{l+1}\Vert\Vert a^{k+1}_l-\overline{a}^{k+1}_l\Vert+\nu\Vert W^{k+1}_{l+1}\Vert\Vert W^{k+1}_{l+1}-\overline{W}^{k+1}_{l+1}\Vert\Vert\overline{a}^{k+1}_l\Vert+\nu\Vert W^{k+1}_{l+1}-\overline{W}^{k+1}_{l+1}\Vert\Vert\overline{W}^{k+1}_{l+1}\Vert\Vert\overline{a}^{k+1}_l\Vert\text{(Cauchy-Schwarz inequality)}
\end{align*}
Because $\Vert W^{k+1}_{l+1}\Vert$, $\Vert \overline{W}^{k+1}_{l+1}\Vert$  and $\Vert \overline{a}^{k+1}_l\Vert$ are upper bounded, $\nu\Vert (W^{k+1}_{l+1})^TW^{k+1}_{l+1}a^{k+1}_l-(\overline{W}^{k+1}_{l+1})^T\overline{W}^{k+1}_{l+1}\overline{a}^{k+1}_l\Vert$ is therefore upper bounded by a combination of $\Vert W^{k+1}_{l+1}-\overline{W}^{k+1}_{l+1}\Vert$ and $\Vert a^{k+1}_{l+1}-\overline{a}^{k+1}_{l+1}\Vert$.\\
\indent For $a^{k+1}_{L-1}$, 
\begin{align*}
    &\nabla_{a^{k+1}_{L-1}}\phi(\textbf{W}^{k+1},\textbf{b}^{k+1},\textbf{z}^{k+1},\textbf{a}^{k+1},u^{k+1})\\&=\tau^{k+1}_{L-1}\circ(a^{k+1}_{L-1}-\overline{a}^{k+1}_{L-1})+\nabla_{\overline{a}^{k+1}_{L-1}}\phi(\textbf{W}_{L-1}^{k+1},\textbf{b}^{k+1}_{L-1},\textbf{z}^{k+1}_{L-1},\textbf{a}^{k+1}_{L-2},u^{k})-\tau^{k+1}_{L-1}\circ(a^{k+1}_{L-1}-\overline{a}^{k+1}_{L-1})-\nabla_{\overline{a}^{k+1}_{L-1}}\phi(\textbf{W}_{L-1}^{k+1},\textbf{b}^{k+1}_{L-1},\textbf{z}^{k+1}_{L-1},\textbf{a}^{k+1}_{L-2},u^{k})\\&+\nabla_{a^{k+1}_{L-1}}\phi(\textbf{W}^{k+1},\textbf{b}^{k+1},\textbf{z}^{k+1},\textbf{a}^{k+1},u^{k+1})\\&=\tau^{k+1}_{L-1}\circ(a^{k+1}_{L-1}-\overline{a}^{k+1}_{L-1})+\nabla_{\overline{a}^{k+1}_{L-1}}\phi(\textbf{W}_{L-1}^{k+1},\textbf{b}^{k+1}_{L-1},\textbf{z}^{k+1}_{L-1},\textbf{a}^{k+1}_{L-2},u^{k})-\tau^{k+1}_{L-1}\circ(a^{k+1}_{L-1}-\overline{a}^{k+1}_{L-1})-\rho (\overline{W}^{k+1}_{L})^T(\overline{W}^{k+1}_{L}\overline{a}^{k+1}_{L-1}+\overline{b}^{k+1}_{L}-\overline{z}^{k+1}_{L}-u^k/\rho)\\&-\nu(\overline{a}^{k+1}_{L-1}-f_{L-1}(z^{k+1}_{L-1}))+\rho (W^{k+1}_{L})^T(W^{k+1}_{L}a^{k+1}_{L-1}+b^{k+1}_{L}-z^{k+1}_{L}-u^{k+1}/\rho)+\nu(a^{k+1}_{L-1}-f_{L-1}(z^{k+1}_{L-1}))
\end{align*}
Because
\begin{align*}
    &\Vert-\tau^{k+1}_{L-1}\circ(a^{k+1}_{L-1}-\overline{a}^{k+1}_{L-1})-\rho (\overline{W}^{k+1}_{L})^T(\overline{W}^{k+1}_{L}\overline{a}^{k+1}_{L-1}+\overline{b}^{k+1}_{L}-\overline{z}^{k+1}_{L}-u^k/\rho)-\nu(\overline{a}^{k+1}_{L-1}-f_{L-1}(z^{k+1}_{L-1}))\\&+\rho (W^{k+1}_{L})^T(W^{k+1}_{L}a^{k+1}_{L-1}+b^{k+1}_{L}-z^{k+1}_{L}-u^{k+1}/\rho)+\nu(a^{k+1}_{L-1}-f_{L-1}(z^{k+1}_{L-1}))\Vert\\&\leq \Vert\tau^{k+1}_{L-1}\circ(a^{k+1}_{L-1}-\overline{a}^{k+1}_{L-1})\Vert+\nu\Vert a^{k+1}_{L-1}-\overline{a}^{k+1}_{L-1}\Vert+\rho\Vert (W^{k+1}_{L})^Tb^{k+1}_{L}-(\overline{W}^{k+1}_{L})^T\overline{b}^{k+1}_{L} \Vert\\&+\rho\Vert (W^{k+1}_{L})^Tz^{k+1}_{L}-(\overline{W}^{k+1}_{L})^T\overline{z}^{k+1}_{L}\Vert+\Vert (W^{k+1}_{L})^Tu^{k+1}-(\overline{W}^{k+1}_{L})^T{u}^{k}\Vert \\&+\rho\Vert (W^{k+1}_{L})^TW^{k+1}_{L}a^{k+1}_{L-1}-(\overline{W}^{k+1}_{L})^T\overline{W}^{k+1}_{L}\overline{a}^{k+1}_{L-1}\Vert\text{(triangle inequality)}
\end{align*}
Then we need to show that $\rho\Vert (W^{k+1}_{L})^Tb^{k+1}_{L}-(\overline{W}^{k+1}_{L})^T\overline{b}^{k+1}_{L} \Vert$, $\rho\Vert (W^{k+1}_{L})^Tz^{k+1}_{L}-(\overline{W}^{k+1}_{L})^T\overline{z}^{k+1}_{L}\Vert$, $\Vert (W^{k+1}_{L})^Tu^{k+1}-(\overline{W}^{k+1}_{L})^T{u}^{k}\Vert$ and $\rho\Vert (W^{k+1}_{L})^TW^{k+1}_{L}a^{k+1}_{L-1}-(\overline{W}^{k+1}_{L})^T\overline{W}^{k+1}_{L}\overline{a}^{k+1}_{L-1}\Vert$  are upper bounded by $\Vert \textbf{W}^{k+1}-\overline{\textbf{W}}^{k+1}\Vert$,  $\Vert\textbf{b}^{k+1}-\overline{\textbf{b}}^{k+1}\Vert$, $\Vert\textbf{z}^{k+1}-\overline{\textbf{z}}^{k+1}\Vert$, $\Vert\textbf{a}^{k+1}-\overline{\textbf{a}}^{k+1}\Vert$ and $\Vert\textbf{z}^{k+1}-\textbf{z}^{k}\Vert$.
\begin{align*}
    &\rho\Vert (W^{k+1}_{L})^Tb^{k+1}_{L}-(\overline{W}^{k+1}_{L})^T\overline{b}^{k+1}_{L} \Vert\\&=\rho\Vert (W^{k+1}_{L})^Tb^{k+1}_{L}-(\overline{W}^{k+1}_{L})^Tb^{k+1}_{L}+(\overline{W}^{k+1}_{L})^Tb^{k+1}_{L}-(\overline{W}^{k+1}_{L})^T\overline{b}^{k+1}_{L} \Vert\\&\leq \rho\Vert b^{k+1}_{L}\Vert\Vert W^{k+1}_{L}-\overline{W}^{k+1}_{L}\Vert+\rho\Vert \overline{W}^{k+1}_{L}\Vert\Vert b^{k+1}_{L}-\overline{b}^{k+1}_{L}\Vert\\&\text{(triangle inequality, Cauchy-Schwarz inequality)}
\end{align*}
Because $\Vert b^{k+1}_{L}\Vert$ and $\Vert\overline{W}^{k+1}_{L}\Vert$ are upper bounded, $\rho\Vert (W^{k+1}_{L})^Tb^{k+1}_{L}-(\overline{W}^{k+1}_{L})^T\overline{b}^{k+1}_{L} \Vert$ is therefore upper bounded by a combination of $\Vert W^{k+1}_{L}-\overline{W}^{k+1}_{L}\Vert$ and $\Vert b^{k+1}_{L}-\overline{b}^{k+1}_{L}\Vert$.\\
Similarly, $\rho\Vert (W^{k+1}_{L})^Tz^{k+1}_{L}-(\overline{W}^{k+1}_{L})^T\overline{z}^{k+1}_{L} \Vert$ is  upper bounded by a combination of $\Vert W^{k+1}_{L}-\overline{W}^{k+1}_{L}\Vert$ and $\Vert z^{k+1}_{L}-\overline{z}^{k+1}_{L}\Vert$.\\
\begin{align*}
    &\Vert (W^{k+1}_L)^Tu^{k+1}-(\overline{W}^{k+1}_L)^Tu^{k}\Vert\\&=\Vert(W^{k+1}_L)^Tu^{k+1}-(W^{k+1}_L)^Tu^{k}+(W^{k+1}_L)^Tu^{k}-(\overline{W}^{k+1}_L)^Tu^{k}\Vert\\&\leq \Vert(W^{k+1}_L)^T(u^{k+1}-u^{k})\Vert+\Vert(W^{k+1}_L-\overline{W}^{k+1}_L)^Tu^{k}\Vert\text{(triangle inequality)}\\&\leq\Vert W^{k+1}_L\Vert\Vert u^{k+1}-u^{k}\Vert+\Vert W^{k+1}_L-\overline{W}^{k+1}_L\Vert \Vert u^{k}\Vert\text{(Cauthy-Schwarz inequality)}\\&=\Vert W^{k+1}_L\Vert\Vert \nabla_{z^{k+1}_L} R(z^{k+1}_L;y)-\nabla_{z^{k}_L}R(z^{k}_L;y)\Vert+\Vert W^{k+1}_L-\overline{W}^{k+1}_L\Vert \Vert u^{k}\Vert \ \text{(Lemma \ref{lemma:z_l optimality})}\\&\leq H\Vert W^{k+1}_L\Vert\Vert z^{k+1}_L-z^{k}_L\Vert+\Vert W^{k+1}_L-\overline{W}^{k+1}_L\Vert\Vert u^{k}\Vert\\&\text{($R(z_L;y)$ is Lipschitz differentiable)}
\end{align*}
Because $\Vert W^{k+1}_L\Vert$ and $\Vert u^k\Vert$ are bounded, $\Vert (W^{k+1}_L)^Tu^{k+1}-(\overline{W}^{k+1}_L)^Tu^{k}\Vert$ is upper bounded by a combination of $\Vert z^{k+1}_L-z^{k}_L\Vert$ and $\Vert W^{k+1}_L-\overline{W}^{k+1}_L\Vert$.
\begin{align*}
    &\rho\Vert (W^{k+1}_{L})^TW^{k+1}_{L}a^{k+1}_{L-1}-(\overline{W}^{k+1}_{L})^T\overline{W}^{k+1}_{L}\overline{a}^{k+1}_{L-1}\Vert\\&=\rho\Vert(W^{k+1}_{L})^TW^{k+1}_{L}a^{k+1}_{L-1}-(W^{k+1}_{L})^TW^{k+1}_{L}\overline{a}^{k+1}_{L-1}+(W^{k+1}_{L})^TW^{k+1}_{L}\overline{a}^{k+1}_{L-1}\\&-(W^{k+1}_{L})^T\overline{W}^{k+1}_{L}\overline{a}^{k+1}_{L-1}+(W^{k+1}_{L})^T\overline{W}^{k+1}_{L}\overline{a}^{k+1}_{L-1}-(\overline{W}^{k+1}_{L})^T\overline{W}^{k+1}_{L}\overline{a}^{k+1}_{L-1}\Vert\\&\leq \rho\Vert(W^{k+1}_{L})^TW^{k+1}_{L}(a^{k+1}_{L-1}-\overline{a}^{k+1}_{L-1})\Vert+\rho\Vert (W^{k+1}_{L})^T(W^{k+1}_{L}-\overline{W}^{k+1}_{L})\overline{a}^{k+1}_{L-1}\Vert+\rho\Vert(W^{k+1}_{L}-\overline{W}^{k+1}_{L})^T\overline{W}^{k+1}_{L}\overline{a}^{k+1}_{L-1}\Vert\text{(triangle inequality)}\\&\leq \rho\Vert W^{k+1}_{L}\Vert\Vert W^{k+1}_{L}\Vert\Vert a^{k+1}_{L-1}-\overline{a}^{k+1}_{L-1}\Vert+\rho\Vert W^{k+1}_{L}\Vert\Vert W^{k+1}_{L}-\overline{W}^{k+1}_{L}\Vert\Vert\overline{a}^{k+1}_{L-1}\Vert+\rho\Vert W^{k+1}_{L}-\overline{W}^{k+1}_{L}\Vert\Vert\overline{W}^{k+1}_{L}\Vert\Vert\overline{a}^{k+1}_{L-1}\Vert\text{(Cauchy-Schwarz inequality)}
\end{align*}
Because $\Vert W^{k+1}_{L}\Vert$, $\Vert \overline{W}^{k+1}_{L}\Vert$  and $\Vert \overline{a}^{k+1}_{L-1}\Vert$ are upper bounded, $\rho\Vert (W^{k+1}_{L})^TW^{k+1}_{L}a^{k+1}_{L-1}-(\overline{W}^{k+1}_{L})^T\overline{W}^{k+1}_{L}\overline{a}^{k+1}_{L-1}\Vert$ is therefore upper bounded by a combination of $\Vert W^{k+1}_{L}-\overline{W}^{k+1}_{L}\Vert$ and $\Vert a^{k+1}_{L}-\overline{a}^{k+1}_{L}\Vert$.\\
\indent For $u^{k+1}$,
\begin{align*}
    \nabla_{u^{k+1}_l} L_\rho&=\nabla_{u^{k+1}_l}\phi(\textbf{W}^{k+1},\textbf{b}^{k+1},\textbf{z}^{k+1},\textbf{a}^{k+1},{u}^{k+1})\\&=z^{k+1}_L-W^{k+1}_La^{k+1}_L-b^{k+1}_L\\&=(1/\rho) (u^{k+1}-u^k)\\&=(1/\rho)(\nabla_{z^{k}_L}R(z^{k}_L;y)-\nabla_{z^{k+1}_L}R(z^{k+1}_L;y))\text{(Lemma \ref{lemma:z_l optimality})}
\end{align*}
Because 
\begin{align*}
    &\Vert(1/\rho)(\nabla_{z^{k}_L}R(z^{k}_L;y)-\nabla_{z^{k+1}_L}R(z^{k+1}_L;y))\Vert\\&\leq (H/\rho)\Vert z^{k+1}_L-z^{k}_L\Vert\text{($R(z_L;y)$ is Lipschitz differentiable)}
\end{align*}
Therefore, $\Vert \nabla_{u^{k+1}_l} L_\rho\Vert$ is upper bounded by $\Vert z^{k+1}_L-z^{k}_L\Vert$.
\end{proof}
\section{Proof of Theorem \ref{thero: theorem 3}}
\label{sec:convergence rate}
\begin{proof}
To prove this theorem, we will first show that $c_k$ satisfies two conditions: (1). $c_k\geq c_{k+1}$. (2). $\sum\nolimits_{k=0}^\infty c_k$ is bounded.  We then conclude the convergence rate of $o(1/k)$ based on these two conditions. Specifically, first, we have
\begin{align*}
    c_k&=\min\nolimits_{0\leq i\leq k}(\sum\nolimits_{l=1}^L(\Vert \overline{W}_l^{i+1}-W_l^i\Vert^2_2+\Vert {W}_l^{i+1}-\overline{W}_l^{i+1}\Vert^2_2+\Vert \overline{b}_l^{i+1}-b_l^i\Vert^2_2+\Vert {b}_l^{i+1}-\overline{b}_l^{i+1}\Vert^2_2) +\sum\nolimits_{l=1}^{L-1}(\Vert \overline{a}_l^{i+1}-a_l^i\Vert^2_2+\Vert {a}_l^{i+1}-\overline{a}_l^{i+1}\Vert^2_2)\\&+\Vert \overline{z}^{i+1}_L-{z}^{i}_L\Vert^2_2+\Vert z^{i+1}_L-\overline{z}^{i+1}_L\Vert^2_2) \\&\geq\min\nolimits_{0\leq i\leq k+1}(\sum\nolimits_{l=1}^L\Vert \overline{W}_l^{i+1}-W_l^i\Vert^2_2+\Vert {W}_l^{i+1}-\overline{W}_l^{i+1}\Vert^2_2+\Vert \overline{b}_l^{i+1}-b_l^i\Vert^2_2+\Vert {b}_l^{i+1}-\overline{b}_l^{i+1}\Vert^2_2) +\sum\nolimits_{l=1}^{L-1}(\Vert \overline{a}_l^{i+1}-a_l^i\Vert^2_2+\Vert {a}_l^{i+1}-\overline{a}_l^{i+1}\Vert^2_2)\\&+\Vert \overline{z}^{i+1}_L-{z}^{i}_L\Vert^2_2+\Vert z^{i+1}_L-\overline{z}^{i+1}_L\Vert^2_2)\\&= c_{k+1}
\end{align*}
Therefore $c_k$ satisfies the first condition. Second,
\begin{align*}
    &\sum\nolimits_{k=0}^\infty c_k\\&=\sum\nolimits_{k=0}^\infty\min\nolimits_{0\leq i\leq k}(\sum\nolimits_{l=1}^L(\Vert \overline{W}_l^{i+1}-W_l^i\Vert^2_2+\Vert {W}_l^{i+1}-\overline{W}_l^{i+1}\Vert^2_2+\Vert \overline{b}_l^{i+1}-b_l^i\Vert^2_2+\Vert {b}_l^{i+1}-\overline{b}_l^{i+1}\Vert^2_2) +\sum\nolimits_{l=1}^{L-1}(\Vert \overline{a}_l^{i+1}-a_l^i\Vert^2_2+\Vert {a}_l^{i+1}-\overline{a}_l^{i+1}\Vert^2_2)\\&+\Vert \overline{z}^{i+1}_L-{z}^{i}_L\Vert^2_2+\Vert z^{i+1}_L-\overline{z}^{i+1}_L\Vert^2_2)\\&\leq \sum\nolimits_{k=0}^\infty(\sum\nolimits_{l=1}^L(\Vert \overline{W}_l^{k+1}-W_l^k\Vert^2_2+\Vert {W}_l^{k+1}-\overline{W}_l^{k+1}\Vert^2_2+\Vert \overline{b}_l^{k+1}-b_l^k\Vert^2_2+\Vert {b}_l^{k+1}-\overline{b}_l^{k+1}\Vert^2_2) +\sum\nolimits_{l=1}^{L-1}(\Vert \overline{a}_l^{k+1}-a_l^k\Vert^2_2+\Vert {a}_l^{k+1}-\overline{a}_l^{k+1}\Vert^2_2)\\&+\Vert \overline{z}^{k+1}_L-{z}^{k}_L\Vert^2_2+\Vert z^{k+1}_L-\overline{z}^{k+1}_L\Vert^2_2) \\&\leq (L_\rho(\textbf{W}^0,\textbf{b}^0,\textbf{z}^0,\textbf{a}^0,u^0)-L_\rho(\textbf{W}^*,\textbf{b}^*,\textbf{z}^*,\textbf{a}^*,u^*))/C_3\text{(Property \ref{pro:property 2})}
\end{align*}
So $\sum\nolimits_{k=0}^\infty c_k$ is bounded and $c_{k}$ satisfies the second condition. Finally, it has been proved that the sufficient conditions of convergence rate $o(1/k)$ are: (1) $c_k\geq c_{k+1}$, and (2) $\sum\nolimits_{k=0}^\infty c_k$ is bounded, and (3) $c_k\geq0$ (Lemma 1.2 in \cite{deng2017parallel}). Since we have proved the first two conditions and the third one $c_k \geq 0$ is obvious, the convergence rate of $o(1/k)$ is proven. 
\end{proof}
\end{appendix}

\end{document}